\documentclass{article}
\usepackage[utf8]{inputenc}
\usepackage{amsmath,amsfonts,amssymb}
\usepackage{a4wide,amsthm}
\usepackage[english]{babel}
\usepackage[babel=true]{csquotes}
\usepackage{textcomp}
\usepackage{comment}
\usepackage{hyperref} 
\usepackage{pdfpages}
\usepackage{pgf,tikz}
\usepackage{enumitem}
\usepackage{mathrsfs}
\usetikzlibrary{arrows}
\usepackage{mathrsfs}
\renewcommand{\div}{\operatorname{div}}
{\newtheorem{thm}{Theorem}[section]}
{\newtheorem{prop}[thm]{Proposition}}
{}
{\newtheorem{lemme}[thm]{Lemma}}
{\newtheorem{defi}[thm]{Definition}}
{\newtheorem{rem}{Remark}[section]}
\newcommand{\Lip}{\text{Lip}}
\newcommand{\LLip}{\text{Log-Lip}}
\newcommand{\supp}{\text{supp} }
\newcommand{\aux}{\text{aux}}

\providecommand{\R}{\mathbb{R}}

\providecommand{\N}{\mathbb{N}}

\renewcommand{\leq}{\leqslant}
\renewcommand{\geq}{\geqslant}

\renewcommand{\div}{\operatorname{div}}

\newcommand{\Id}{\operatorname{Id}}
\newcommand{\loc}{\operatorname{loc}}

\newcommand{\eps}{\epsilon}

\newcommand{\Addresses}{{
  \bigskip
  \footnotesize

\noindent  A.~Mecherbet, \textsc{Institut de Mathématiques de Jussieu-Paris Rive Gauche, Université Paris Cité,
    8~Place Aurélie Nemours,
F75205 Paris Cedex~13, France}\par\nopagebreak
\noindent  \textit{E-mail address: }\texttt{mecherbet@imj-prg.fr}

  \medskip

\noindent  F.~Sueur, \textsc{  Institut de Math\'ematiques de Bordeaux, UMR CNRS 5251,Universit\'e de Bordeaux, 351~cours
de la Lib\'eration, F33405 Talence Cedex, France   $\&$ Institut  Universitaire de France}\par\nopagebreak 
\noindent  \textit{E-mail address:} \texttt{Franck.Sueur@math.u-bordeaux.fr}

}}

\title{A few remarks on the transport-Stokes system}

\author{Amina Mecherbet and Franck Sueur}

\date{\today}

\begin{document}

\maketitle

\abstract{We consider the so-called transport-Stokes system which describes sedimentation of inertialess suspensions in a viscous flow and couples a transport equation and the steady Stokes equations in the full three-dimensional space. 
 First we present a global existence and uniqueness result for $L^1 \cap L^p$ initial densities where $p \geq 3$.
Secondly, we prove that, in the case where $p>3$, 
the flow map which describes the trajectories of these solutions is analytic with respect to time. 
Finally we establish the   small-time global  exact   controllability of    the transport-Stokes system.
These results  extend to the transport-Stokes system some results obtained for the incompressible Euler system respectively by Yudovich in \cite{Yudovich63}, by Chemin in \cite{Chemin92,Chemin} and by Coron, and Glass, in \cite{Coron96,G20}.  }

\tableofcontents

\section{Introduction and earlier results}
We consider the following transport-Stokes  system:
\begin{equation}
\label{TS}
\left\{
\begin{array}{rcll}
\partial_t \rho+{\div}( u \rho ) &=& 0\,,& \text{on $\mathbb{R}^+ \times \mathbb{R}^3,$}\\
- \Delta u+ \nabla q &=& - \rho e_3 \,,& \text{on $\mathbb{R}^+ \times \mathbb{R}^3$},\\
\div u & =&0\,,& \text{on $\mathbb{R}^+ \times \mathbb{R}^3$},\\
\underset{|x|\to \infty}{\lim}|u| &=& 0,& \text{on } \mathbb{R}^+,\\
\rho(0,\cdot) &= & \rho_0 \,,&\text{on $ \mathbb{R}^3$}.
\end{array}
\right.
\end{equation}
The system \eqref{TS} has been derived in \cite{Hofer, Mecherbet} as a model for the sedimentation of a cloud of rigid particles in a viscous fluid in the case where inertia of both fluid and particles is neglected. Let us also mention that, recently,  some systems coupling a transport equation and the Navier-Stokes equations have been obtained as 
hydrodynamic limits of the Vlasov-Navier-Stokes system in \cite{HKM}.

Existence and uniqueness of solutions to \eqref{TS} 
has been proved in \cite{Hofer} for regular initial data density, and extended to the case where the initial density is in $L^1(\R^3) \cap L^\infty(\R^3)$  with finite first moment  in   \cite{Mecherbet2}. A similar result, without the assumption on the finite first moment, has been proved in the parallel contribution \cite{Hofer&Schubert}. In the former paper, the authors also present a well-posedness result for a transport-Stokes like model taking into account the correction to the effective viscosity, the initial probability density is assumed to be $W^{1,\infty}(\R^3)\cap W^{1,1}(\R^3) $ in this case. In the paper \cite{Leblond}, the author proves global existence and uniqueness for $L^\infty(\R^3)$ initial densities in the case of bounded domain in $\R^3$ and $\R^2$ and also in the case of an infinite strip $\Omega= (0,1)\times \R$ with a flux condition.  We refer also to the recent paper \cite{GrayerII} where the author proves global existence and uniqueness in $\R^2$ for $L^1(\R^2)\cap L^\infty(\R^2)$ compactly supported initial densities as well as propagation of H\"older regularity of the boundary of the patch.

Let us finally mention the work in progress \cite{Cobb} regarding the 
wellposedness of the transport-fractional Stokes system. 

Let us emphasize that the transport-Stokes system can be seen as a particular case of a large family of transport equations with a velocity given as the convolution of the density $\rho$ with a kernel $K$. The main key in order to ensure uniqueness is to establish suitable stability estimates. The first result relying on the 1-Wasserstein distance was obtained by Dobrushin \cite{Dobrushin} in the case where $\rho \in \mathcal{P}_1(\R^3)$ and $K\in Lip(\R^3,\R^3)$. This result has been generalized since for more singular kernels but for which, in counterpart, the probability density $\rho$ lies in $L^\infty(\R^3)$; we refer to the stability result by Hauray \cite{Hauray} in the case where the singular kernel (at origin) satisfies a compatibility condition that mainly ensures local integrability of $K$, $\nabla K$. This result has been adapted from the 2-Wasserstein stability estimate obtained by Loeper \cite{Loeper} for the Vlasov-Poisson system with $L^\infty$ probability measures. We finally refer to the paper by Duerinckx and Serfaty \cite{Serfaty&Duerinckx} where $K$ is the gradient of Coulomb, logarithmic or Riesz kernel. Authors prove the mean field convergence for a large system of interacting particles by introducing a suitable distance built as Coulomb (or Riesz) metric which overcomes the stability issue encountered when considering the Coulomb case with Wasserstein metric.

\section{Statement of the main  results}

\subsection{Global existence and uniqueness result for \texorpdfstring{$L^1 \cap L^p$}{L1 n Lp} initial densities  where \texorpdfstring{$p \geq 3$}{p>=3}.}
In this section, we state the first main result of this paper regarding the global existence and uniqueness of solutions to the 
transport-stokes system with lower regularity than in the previous results, that is  for $L^1 \cap L^p$ initial densities where $p \geq 3$.

Let us fix a notation before to state precisely the statement. 
For $p\geq 1$, let  $E_p$ be the subspace of probability measures
in $ L^p (\R^3) $.
In addition to the global existence and uniqueness issue, next result also establishes the validity of some well-known  conservation laws and the existence of a flow map in a classical sense. 
\begin{thm} \label{thm0}
Let $p\geq 3$ and $\rho_0$ in $E_p$. Then there is a unique corresponding solution 
 $$(u,q,\rho)\in
 \left \{ 
 \begin{array}{lcl}
 C([0,+\infty),W^{2,p}  (\R^3) \times W^{1,p}  (\R^3)\times  E_p) & \text {if} & p>3 ,\\ C([0,+\infty),\dot{W}^{2,3}  (\R^3) \cap \dot{W}^{1,3}  (\R^3)\cap  \underset{q\in(3,+\infty]}{\bigcap} W^{1,q}(\R^3) \times W^{1,3}  (\R^3)\times  E_p) & \text {if} & p=3,
 \end{array}
 \right. $$ 
 of the transport-Stokes equation \eqref{TS} in the sense of distributions, the density function satisfies the following conservation law: for any $t\geq 0$,
\begin{eqnarray}\label{conserLp}
\|\rho(t)\|_{ L^1 (\R^3) } =\|\rho_0\|_{ L^1 (\R^3) }  \,   \text{  and } \, \|\rho(t)\|_{L^p} =\|\rho_0\|_{ L^p}  .
\end{eqnarray}
The velocity and pressure $(u,q)$ satisfy for $p>3$ and all $t\geq 0$
\begin{equation}\label{reg_u,p}
\|u(t)\|_{W^{2,p}(\R^3)}+ \|q(t)\|_{W^{1,p}(\R^3)} \leq C_p \| \rho(t) \|_{L^1(\R^3)\cap L^p(\R^3)},
\end{equation}
and for $p=3$, $q\in(3,+\infty]$
\begin{equation}\label{reg_u,p_bis}
\|\nabla u(t)\|_{W^{1,3}(\R^3)}+\|u(t)\|_{W^{1,q}(\R^3)}+ \|q(t)\|_{W^{1,3}(\R^3)} \leq C_q \| \rho(t) \|_{L^1(\R^3)\cap L^3(\R^3)}.
\end{equation}
Moreover,
\begin{equation}\label{ulipoullip}
    u \in 
     \begin{cases} 
C([0,+\infty),  C^{1,\mu} (\R^3),   &  \text{  if } \, p > 3,
      \\  C([0,+\infty), \LLip (\R^3) ),   &  \text{  if } \, p = 3.
\end{cases}
\end{equation}
where $\mu=1-3/p$,
see Definition \ref{def:LLspace} for more details on Log-Lipschitz spaces. Finally, for all $s\in[0,+\infty)$, there exists a unique 
\begin{equation}\label{uflox}
    X(\cdot,s,\cdot)-\Id \in 
     \begin{cases} 
C([0,+\infty),  C^{1,\mu} (\R^3) ),   &  \text{  if } \, p > 3,
      \\  {C([0,+\infty), C^{0,r_t} (\R^3) ),}   &  \text{  if } \, p =3, 
\end{cases}
\end{equation}
where $r_t = e^{-Ct} $, with $C>0$ a constant depending only on $\|\rho_0\|_{L^1(\R^3)\cap L^3(\R^3)}$,
such that 
\begin{equation} \label{cflow-def}
\left \{
\begin{array}{rcll}
\partial_t Xt,s,x) & = & u(t, X(t,s,x)),&\: \: \forall \, t,s\in[0,+\infty) , \\
X(s,s,x) & = & x,& \: \: \forall \, s \in [0,+\infty) ,
\end{array}
\right.
\end{equation}
and for any $t$ in $[0,+\infty)$, 
\begin{equation}
    \label{pousse-th}
    \rho(t,\cdot) =X(t,0,\cdot) \#\rho_0 .
    \end{equation}
 \end{thm}
 In \eqref{pousse-th} the notation 
$\#$ stands for the push-forward of the measure which follows the symbol  by the mapping which precedes the symbol.

 \begin{rem}\label{analog}
The critical case where  $p=3$ can be seen as a counterpart of the famous result by Yudovich in 
\cite{Yudovich63} on the existence and uniqueness of the weak solutions with bounded vorticity of the 2D  incompressible Euler equations. 
In particular, this result uses in a key manner that the velocity field is log-Lipschitz. A related result, with a conditional flavour, also holds in the case of the Vlasov-Poisson equations, see \cite{Loeper}. 
\end{rem}
The most delicate part of the proof of Theorem \ref{thm0} is the uniqueness issue. 
 Indeed Theorem \ref{thm0} is obtained by classical arguments from the following  stability estimates where  $W_1$  denotes the first Wasserstein distance defined for two  probability measures $\rho_1,\rho_2 $ 
by 
$$ W_1 (\rho_1,\rho_2) := 
\underset{\pi \in \Pi(\rho_1,\rho_2)} \inf \int |x-y| d\pi(x,y)  , $$
 where $\Pi(\rho_1,\rho_2)$ is the set of admissible transport plans $\pi $ in $\mathcal{P}(\mathbb{R}^3\times \mathbb{R}^3)$
 having $\rho_1$ as first marginal and $\rho_2$ as second marginal.
\begin{thm} \label{thm}
Let $p\geq 3$.
For any $R>0$, there is $C>0$ such that for any couple of initial densities 
  $ (\rho_1^0 , \rho_2^0 )$   in $E_p $ with 
  $ \max_{i=1,2} (\|\rho_i^0\|_{ L^1 (\R^3) },\|\rho_i^0\|_{ L^p (\R^3) }) \leq R$, if  $ (u_1,\rho_1) $ and $ (u_2 ,\rho_2)$ satisfy the transport-Stokes equation \eqref{TS} 
  for any $t\geq 0$, then 
     \begin{align}
         \label{sta-geq-3}
         W_1(\rho_1(t),\rho_2(t)) &\leq W_1(\rho_1^0,\rho_2^0)e^{C t}, \quad  \text{if } p > 3, 
              \\    \label{sta-geq-critik}
         W_1 (\rho_1(t),\rho_2(t)) &\leq W_1(\rho_1^0,\rho_2^0)^{e^{-C t}} , \quad  \text{if } p = 3 .
     \end{align}
 \end{thm}
The proof of  Theorem \ref{thm} is given in Section 
 \ref{sec-stabb}
 and 
Theorem \ref{thm0} is deduced from  Theorem \ref{thm} in Section 
\ref{sec-proof-princi}. 
  \begin{rem}\label{rem:TS_avec_terme_source}
  We emphasize that the existence and uniqueness result can be extented to the case of transport-Stokes system with an additional source term $f$ in the right hand side of the Stokes equation provided that $f \in  C([0,+\infty), L^1(\R^3) \cap L^p(\R^3))$.
  \end{rem}

\subsection{Analyticity of the flow map when \texorpdfstring{$p > 3$}{p>3}.}

Next, we prove that, in the case where $p>3$, 
the flow map which describes the trajectories of these solutions is analytic with respect to time. More precisely the following result holds true. 
\begin{thm}
\label{analytic}
Let $p>3$, $\rho_0$ in $E_p$ and 
$ (u,\rho)$ in $C([0,+\infty),W^{2,p} \times  E_p)$
the unique corresponding solution  of the transport-Stokes equation \eqref{TS}.
Let $\mu:=1-3/p$. 
Let $X$  the flow associated with $u$ by \eqref{uflox} and 
\eqref{cflow-def}. 
Then $ X$ is analytic from $[0,+\infty)$ to $ C^{1,\mu}(\R^3)$. 
\end{thm}
This result  extends to the transport-Stokes system a result which was obtained for classical solutions of the incompressible Euler system by Chemin in \cite{Chemin}, and reproved by various other methods see  by, among others \cite{CVW,FZ,Gamblin,Inci,Kato,ogfstt,Serfati,Shnirelman}.

To prove Theorem \ref{analytic}, the key point is to prove the following result. 
\begin{prop} \label{apana}
There exists $T>0$ depending only on  $\|\rho_0\|_{L^1 \cap L^p}$ and two constants $C_1 $ and $C_0$ depending only on $\|\rho_0\|$, $T$ and $\mu$  such that for all $t\in[0,T]$ and for all $n\geq1$,
\begin{equation}\label{rec_formule}
\|\partial_t^{n} X(t,\cdot)\|_{1,\mu} \leq (-1)^{n-1} n! \begin{pmatrix}1/2 \\ n  \end{pmatrix}C_0^n C_1^{n-1}.
\end{equation}
\end{prop}
 Above we have used, for any $n\geq 1$, the notation 
\begin{equation}\label{binomial_1/2}
    \begin{pmatrix}
    1/2\\
    n
    \end{pmatrix}= \frac{(1/2)(1/2-1) \cdots (1/2-n+1)}{n!};
\end{equation}
which is extended by convention by 
$\begin{pmatrix}
1/2\\0
\end{pmatrix}=-1$. 
With this definition we observe that 
$$
(-1)^{n-1}\begin{pmatrix}
1/2\\n
\end{pmatrix}\geq 0.
$$
Moreover, for all $n\geq 2$,
\begin{equation}\label{binomial_1/2_borne}
  (-1)^{n-1}  n!\begin{pmatrix}
1/2\\n
\end{pmatrix} \leq C \frac{n!}{2^n}.
\end{equation}
Therefore to deduce Theorem \ref{analytic} from  Proposition \ref{apana}, 
it is sufficient to combine \eqref{rec_formule} and \eqref{binomial_1/2_borne}
on any time interval of the form $[t,t+T] $. Indeed, since the ODE equation for the flow is autonomous, we have stated Proposition \ref{apana} in the case  where $t=0$ for simplicity.

 The proof of Proposition \ref{apana} is given in Section 
 \ref{sec-apana} and is based on the approach developed in \cite{CVW} where  authors prove analyticity of Lagrangian trajectories for several incompressible  inviscid models including the 2D surface quasi-geostrophic equation, the 2D incompressible porous medium equation, the 2D and the 3D incompressible Euler equations  and the 2D Boussinesq equations. It turns out that, for the transport-Stokes model, the  Oseen kernel benefits from better integrability properties, compared to the kernels considered in \cite{CVW}, which allows us to consider an initial density function $\rho_0$ which is less regular than the initial data considered in \cite{CVW}.

 \begin{rem}\label{G3}
In the case where $p=3$, it is likely that 
 the flow  map  $X$   still benefits from Gevrey regularity from $[0,+\infty)$ to $C^{r_t}(\R^3)$.
 Such a result is proved in the case of the $2D$ incompressible Euler equations with bounded vorticity in \cite{Gamblin} in the case where the fluid occupies the full space and in 
 \cite{Sueur} in the case where  the fluid occupies  a bounded domain. 
\end{rem}

\subsection{Exact controllability of  the transport-Stokes system}

Our third main result establishes the   small-time global  exact controllability of  the transport-Stokes system 
when distributed forces are allowed in an arbitrary non-empty open subset of $\R^3$, in the case where the initial and final densities are in  $L^p_c$ with $p \geq 3$.
Here the index $c$ means that these densities are compactly supported. 

\begin{thm}\label{TS-control}
Let $T>0$, $p\geq 3$. Let $\omega$ a non-empty open subset of $\R^3$. 
Let $\rho_0$ and $\rho_f$ in $L^p_c$.
There exists $ (u,\rho)$ in $C([0,T),W^{2,p}(\R^3) \times  L^p(\R^3))$ and analogously for $p=3$, $u \in C([0,T],\dot{W}^{1,3}(\R^3)\cap \dot{W}^{2,3}(\R^3) \cap \bigcap_{q>3} L^q(\R^3)$, there are $f \in C_c(0,T, L^\infty(\R^3))$ and $g \in C_c(0,T,L^1(\R^3)\cap L^p(\R^3)) $, both compactly supported in $(0,T)\times \omega $ such that 
\begin{equation*}
\left \{
\begin{array}{rcll}
  \partial_t \rho +\div(\rho u)&=& g,& \text{ on }  \R^3 ,\\
    -\Delta u+\nabla p &=& - \rho e_3+f ,& \text{ on } \R^3 ,\\
    \div u&=& 0,& \text{ on } \R^3 ,\\
    \rho(0)=\rho_0&,& \rho(T)=\rho_f,& \text{ on } \R^3 .
\end{array}    
    \right.
\end{equation*}
\end{thm}
Theorem \ref{TS-control} adapts to the case of the transport-Stokes system  some earlier results obtained for the incompressible Euler equations in \cite{Coron96,G20,G}.
The proof follows closely a now well-known strategy, relying on Coron's return method  and the existence of peculiar vector fields, which solves the Stokes system in the  uncontrolled zone $\R^3 \setminus \omega$. For sake of completeness, we give a sketch of proof in Section 
\ref{sec-cont}. 
Let us also mention that results regarding the Lagrangian controllability could also be obtained, in the spirit of the  results   \cite{GH1,GH2,HK}  in the case of the Euler equations and the result \cite{GH3} in the case of the steady Stokes equations.
The proof of Theorem \ref{TS-control} is given in Section
\ref{sec-control}. 

\subsection{A few open questions on the transport-Stokes system}
\begin{itemize}
    \item 
It would be interesting to investigate whether the results obtained in this paper can be extended to the case of $\R^2$ and the case of domains with boundaries. We have already mentioned the recent paper \cite{GrayerII} for the case of compactly supported initial densities in $L^1(\R^2)\cap L^\infty(\R^2)$ and the paper \cite{Leblond} where existence and uniqueness is proved in the case of bounded domain in $\R^3$ and $\R^2$ and also in the case of an infinite strip $\Omega= (0,1)\times \R$ with a flux condition. In both cases the density is assumed to be in $ L^{\infty}$.  Regarding the issue of the regularity in time of the flow map, let us mention \cite{Kato,ogfstt,GS,Sueur} where the case 
of the incompressible Euler equations in  domains with boundaries is tackled.

 \item Another question is whether or not the well-posedness part of our results could be extended to lower regularity by making use of the theory of renormalized or Lagrangian solutions in the spirit, for example, of 
\cite{Lions,CCS}. One natural goal in this direction would be to be able to deal with densities  $\rho$  which are merely measures, with the aim to tackle 
 the mean field limit. 
\item  It has been shown in \cite{Mecherbet2} that if $\rho_0=1_{B_0}$ with $B_0$ the unit ball then the solution $\rho=1_{B}$ remains a spherical patch for all time, more precisely  there exists a constant $v\in \R^3$ such that $(\rho,u)$ satisfies
\begin{equation}\label{eq:hadamard_Rybczynski}
\begin{array}{rcl}
(u(t,x),\rho(t,x))&=&(u_0(x-vt),\rho_0(x-vt)),\\
(u-v)\cdot n& =& 0 \text{ on } \partial B ,
\end{array}
\end{equation}
where $n$ is the unit outer normal vector. A natural question is then to investigate the reciprocal property, that is, if there exists a constant velocity $v$ and a bounded domain $B_0$ for which \eqref{eq:hadamard_Rybczynski} is satisfied then $B_0$ is necessarily a ball. Note that this type of characterisation of the ball has been proved in any dimension by \cite[Theorem 1.1]{Fraenkel} in the case where $u_0$ is a Newtonian potential housed in a bounded open set $G$. The author shows that if $u_0$ is constant on $\partial G$ then necessarily $G$ is a ball.

 A related issue would be to investigate the spherical case in the two dimensional setting as well since well posedness of $2$d patchs were established in \cite{GrayerII}.

\end{itemize}


\section{A few preliminary reminders}

In this section we gather some classical material which is useful later on.

\subsection{Multivariable calculus tools}

We denote by $\mathbb{N}_0^3$ the set of three dimensional multi-indices $\alpha=(\alpha_1,\alpha_2,\alpha_3)$ endowed with the component-wise sum and difference and with a partial order 
$$
\alpha \leq \beta \Leftrightarrow \Big( \alpha_i \leq \beta_i ,\quad \forall 1 \leq i \leq 3 \Big).
$$
We introduce the following notations: for any $\alpha,\beta \in \mathbb{N}_0^3$, for any $y\in\R^3$, we set
\begin{gather*}
|\alpha|:= \alpha_1+\alpha_2+\alpha_3 ,\quad
\alpha! := \alpha_1! \alpha_2!\alpha_3! ,\quad
\partial^\alpha :=\partial_{x_1}^{\alpha_1}\partial_{x_2}^{\alpha_2} \partial^{\alpha_3} ,\quad
y^\alpha :=\left(y_1^{\alpha_1}\right) \left( y_2^{\alpha_2} \right) \left(y_3^{\alpha_3}\right) ,\\
\begin{pmatrix}
\alpha \\
\beta 
\end{pmatrix}= \begin{pmatrix}
\alpha_1 \\
\beta_1 
\end{pmatrix}\begin{pmatrix}
\alpha_2 \\
\beta_2 
\end{pmatrix}\begin{pmatrix}
\alpha_3 \\
\beta_3
\end{pmatrix}= \frac{\alpha!}{\beta!(\alpha-\beta)!},
\end{gather*}
and we recall the classical Leibniz formula 
$$
\partial^\alpha (fg) = \underset{\beta\leq\alpha}{\sum} (\partial^\beta f) (\partial^{\alpha-\beta} g).
$$
We introduce for all $n\geq1$,  $1\leq s \leq n$ and $\alpha \in \mathbb{N}_0^3$ with $1 \leq |\alpha|\leq n$ the set
$P_s(n,\alpha)$ of the 
$$(k_1,\cdots,k_s;l_1,\cdots,l_s)\in \N_0^3\times \cdots\times \N_0^3 \times \N \times \cdots \N, $$
such that 
$$0<|k_j|,\quad 0<l_1<\cdots<l_s,\quad \underset{j=1}{\overset{s}{\sum}}k_j=\alpha,\quad \underset{j=1}{\overset{s}{\sum}}|k_j| l_j=n .$$

We  now recall the following formula from \cite[Lemma 3.2]{CVW}, which is useful for the proof of Proposition \ref{apana}. 
\begin{lemme}[Multivariate Fa\`a di Bruno formula]\label{Lemme_FDB}
Let $g:\R \mapsto \R^3$ be a vector function $ C^\infty$ in the neighborhood of $x_0 \in \R$.
Let $h : \R^3 \to \R$ be a scalar function $ C^\infty$ in the neighborhood of $y_0=g(x_0)$. Define $f(x):=h(g(x)): \R \to \R$. Then
\begin{equation}\label{FDB_formula}
 f^{(n)}(x_0)=  n! \underset{1\leq |\alpha|\leq n}{\sum} (\partial^\alpha h)(g(x_0)) \underset{s=1}{\overset{n}{\sum}}\underset{P_s(n,\alpha)}{\sum}\underset{j=1}{\overset{s}{\prod}}\frac{(\left(\partial^{l_j} g)(x_0)\right)^{k_j}}{(k_j!) (l_j!)^{|k_j|}} ,
\end{equation}
holds for any $n\geq 1$, with the convention that $0^0:=1$.
\end{lemme}

\subsection{H\"older,  Lipschitz and  log-Lipschitz spaces}
%
We introduce the following notations for Lipschitz, log-Lipschitz and H\"older spaces.
\begin{defi}[H\"older spaces]
Let $ \mu \in ]0,1]$ and $n\in \mathbb{N} $
$$
 C^{n,\mu}(\R^3):= \{ \psi \in  C^n(\mathbb{R}^3), \sum_{|\alpha|\leq n}\underset{x \in \R^3}{\sup}|\partial^\alpha \psi(x)|+
\sum_{|\alpha|= n}[\partial^\alpha \psi]_{0,\mu}<+\infty \},
$$
\end{defi}
where 
$$ [\psi]_{0,\mu} := \underset{x \neq y}{\sup}\, \frac{|\psi (x)-\psi (y)|}{|x-y|^\mu}.$$
The space $ C^{n,\mu}(\R^3)$ is endowed with the following norm
$$
\|\psi\|_{n,\mu}:= \underset{i=0}{\overset{n}{\sum}}\underset{x \in \R^3}{\sup}|\partial_i \psi(x)|+ [\partial_n \psi]_{0,\mu}.
$$
In the case of the space of Lipschitz functions $ C^{0,1}(\R^3) $ we use the notation 
$$
\Lip(\psi):= \underset{x \neq y}{\sup}\, \frac{|\psi(x)-\psi(y)|}{|x-y|},
$$
and set 
$$
\Lip(\R^3):= C^{0,1}(\R^3).
$$
We introduce the following real function $\ln_{-}$
\begin{equation}\label{def_ln-}
    \ln_{-}(s):=\max(0,-\ln(s))=\left\{
    \begin{array}{rcl}
    -\ln(s) & if & s\in[0,1], \\
    0 & if & s \geq 1 .
    \end{array}
    \right.
\end{equation}
\begin{defi}[Log-Lipschitz space]\label{def:LLspace}
We introduce the following notation for any function $\psi$

$$
\LLip(\psi):= \underset{x \neq y}{\sup} \frac{|\psi(x)-\psi(y)|}{|x-y|(1+\ln_{-}|x-y|)}.
$$
We say that a bounded continuous function $\psi$ is Log-Lipschitz if $\LLip(\psi)<+\infty$ and we set 
$$
\LLip(\R^3):=\{\psi \in  C(\R^3), \|\psi\|_\infty+ \LLip(\psi)<+\infty \}.
$$
\end{defi}

\subsection{Steady Stokes equations}
\label{SSE}
Let $(\mathcal{U},\mathcal{P})$ the Stokes Green function, also called the Oseen tensor,  given by 
\begin{equation}\label{Oseen}
\mathcal{U}(x)= \frac{1}{8\pi |x|} \left(Id + \frac{x \otimes x}{|x|^2} \right) \ \text{ and } \  \mathcal{P}(x)=- \frac{1}{4\pi} \frac{x}{|x|^2} .
\end{equation}
Let us observe that the Oseen tensor
$\mathcal{U}$ satisfies the following bound: 
\begin{equation}
    \label{B-prop}
    \left|\mathcal{U}(x) \right|\leq C  \frac{1}{|x|}, \: \: \forall x \in \R^3 \setminus \{0\} ,
\end{equation}
and  
\begin{equation}
    \label{Lip-prop}
    \left|\mathcal{U}(x)-\mathcal{U}(y) \right|\leq C |x-y| \left(\frac{1}{|x|^2}+ \frac{1}{|y|^2} \right), \: \: \forall x, y \in \R^3 \setminus \{0\} .
\end{equation}

Let us recall the following classical result on the Stokes equations, see for example  \cite[Theorem IV.2.1]{Galdi}.
\begin{prop}\label{class-Stokes}
Let $f $ in $L^p (\R^3) \cap L^1 (\R^3) $ with $p\geq  3$.
Then 
there is a unique solution of the Stokes equation
\begin{equation}
\label{cstokes}
\left\{
\begin{array}{rcll}
- \Delta u+ \nabla \pi &=& f \,,& \text{in $ \mathbb{R}^3$},\\
\div u & =&0\,,& \text{in $ \mathbb{R}^3$},\\
\underset{|x|\to \infty}{\lim}|u| &=& 0 .
\end{array}
\right.
\end{equation}
such that  $(u,\pi)\in W^{2,p} (\R^3)  \times W^{1,p} (\R^3) $ if $p>3$ and  $(\nabla u,\pi)\in W^{1,3} (\R^3)  \times W^{1,3} (\R^3) $, $u\in W^{1,q}(\R^3)$ for any $q>3$ in the case $p=3$.
Moreover,  
\begin{equation}
    \label{formu-conv}
    u=\mathcal{U} \star f \text{ and } \pi=\mathcal{P} \star f .
\end{equation}
\end{prop}
%



\subsection{Sobolev embedding}
The reason why the value $p=3$ appears in our analysis is due to the following result, which establishes that the unique velocity solution  of the Stokes equation 
with data in $ L^1\cap L^p$ is Lipschitz when $p>3$ and  Log Lipschitz in 
 the critical case $p=3$.

\begin{prop}\label{prop_log_lip_u}
Let $\rho$ in $L^p (\R^3)  \cap L^1 (\R^3) $ with $p\geq 3$, and 
$(u,p) $  the  unique corresponding solution in $W^{2,p} (\R^3) \times W^{1,p} (\R^3) $ (resp. in $\dot{W}^{2,3} (\R^3)\cap \dot{W}^{1,3}(\R^3)\cap W^{1,q}(\R^3) \times W^{1,3} (\R^3) $ for $p=3$, $q\in(3,+\infty]$)
of the Stokes equation \eqref{cstokes}. 
Then 
\begin{equation}\label{uli}
    u \in 
     \begin{cases} 
\Lip (\R^3) ,   &  \text{  if } \, p > 3,
      \\  \LLip (\R^3) ,   &  \text{  if } \, p = 3,
\end{cases}
\end{equation}
with the following estimates: 
\begin{align} \label{esti-uli} 
    \Lip(u) \leq C \| \rho\|_{L^1 (\R^3)\cap L^p (\R^3)} ,   &  \text{  if } \, p > 3,
      \\  \label{esti-ulli} 
      \LLip (u) \leq C \| \rho\|_{L^1 (\R^3)\cap L^3 (\R^3)} &  \text{  if } \, p = 3 .
\end{align}
\end{prop}

\begin{proof}
In the case where  $p>3$ we have the embedding $W^{2,p} (\mathbb{R}^3) \hookrightarrow W^{1,\infty} (\mathbb{R}^3)$, so that 
\begin{eqnarray} \label{cotestokes}
\|u\|_{W^{1,\infty} (\R^3)}\leq C\|u\|_{W^{2,p} (\R^3)} \leq C \| \rho\|_{L^1 (\R^3)\cap L^p (\R^3)}.
\end{eqnarray}

Let us now deal with the case where $p=3$.

We provide here a self-contained proof which we will adapt later on in the proof of 
Theorem \ref{thm}. 
First, by the convolution formula  \eqref{formu-conv}, we have 
\begin{align*}
   |u(x)| &\leq C \int_{B(x,1)} \frac{\rho(z)}{|x-z|} dz + \int_{B(x,1)^c}\frac{\rho(z)}{|x-z|} dz 
   \\  &\leq  \left(\int_{B(x,1)}\frac{dz}{|x-z|^2} \right)^{1/2} \|\rho\|_{L^2(\R^3)}+\|\rho\|_{L^1(\R^3)}
  \\  &\leq  C \|\rho\|_{L^1(\R^3) \cap L^3(\R^3)}, 
\end{align*}
where the notation $\cdot^c$ stands for the complementary set in $\R^3$.

Second,  by the convolution formula  \eqref{formu-conv},  we obtain that for
any  $x\neq x'$ in $\mathbb{R}^3$,
\begin{align*}
|u(x)-u(x') |&\leq 
I_1 +I_2 ,
\end{align*}
where
\begin{align*}
I_1 &:= \int_{B(x,\epsilon) \cup B(x',\epsilon) }  \left|\mathcal{U}(x-u)-\mathcal{U}(x'-u)\right|\rho(u)du , \\
I_2 &:= \int_{B(x,\epsilon)^c \cap B(x',\epsilon)^c} \left|\mathcal{U}(x-u)-\mathcal{U}(x'-u)\right|\rho(u)du ,
\end{align*}
with $\epsilon :=|x-x'|$. 

For $I_1$, by \eqref{B-prop}, 
we have 
\begin{align*}
I_1& \leq C \Big( \int_{B(x,\epsilon)} \frac{1}{|x-u|}\rho(u)du + \int_{B(x',\epsilon)} \frac{1}{|x-u|}\rho(u)du
\\ &\quad + \int_{B(x,\epsilon)} \frac{1}{|x'-u|}\rho(u)du + \int_{B(x',\epsilon)} \frac{1}{|x'-u|}\rho(u)du \Big) .
\end{align*}
Since  
$B(x,\eps) \subset B(x',2\eps) $ and $ B(x',\eps) \subset B(x,2\eps)$ by the triangle inequality and the definition of $\epsilon$, 
 we have   that 
$$
I_1 \leq C \left( \int_{B(x,2\epsilon) } \frac{1}{|x-u|}\rho(u)du + \int_{B(x',2\epsilon)} \frac{1}{|x'-u|}\rho(u)du \right) \leq C \epsilon \|\rho\|_{L^3 (\R^3)}  ,
$$
by H\"older's inequality.

For $I_2$, by \eqref{B-prop}, 
we have
\begin{align*}
I_2 &\leq C |x-x'| \int_{B(x,\epsilon)^c \cap B(x',\epsilon)^c}    \left(\frac{1}{|x-u|^2}+\frac{1}{|x'-u|^2} \right)\rho(u) du\\
&\leq C |x-x'|\left( \int_{B(x,\epsilon)^c} \frac{1}{|x-u|^2} \rho(u) du+ \int_{B(x',\epsilon)^c} \frac{1}{|x'-u|^2} \rho(u) du\right) .
\end{align*}
If $\epsilon=|x-x'|\geq 1$ we have 
$I_2\leq C |x-x'|\| \rho\|_1$ and this yields the desired result. 
On the other hand, if $\epsilon<1$, then 
\begin{align*}
 \int_{B(x,\epsilon)^c} \frac{1}{|x-u|^2} \rho(u) du &\leq \int_{B(x,2)\setminus B(x,\epsilon) }   \frac{1}{|x-u|^2} \rho(u) du +  \int_{B(x,2)^c} \frac{1}{|x-u|^2} \rho(u) du\\
 &\leq C\|\rho\|_{L^3 (\R^3)} | \log(\epsilon)|+ \|\rho\|_{L^1 (\R^3)} ,
\end{align*}
by H\"older's inequality.
By combining the estimates we conclude the proof of \eqref{esti-ulli}. 
\end{proof}


\subsection{Flow map and Osgood's lemma}
The following result on the existence and uniqueness of a classical flow in the case of Lipschitz or Log-Lipschitz vector fiels is quite classical, let us refer for example to \cite{BCD}. 
\begin{prop}\label{prop_flow}
Let   $u$ a vector field in \eqref{ulipoullip}. 
 Then there exists a unique corresponding flow
$X$ such as in \eqref{uflox} and satisfying \eqref{cflow}. 
\end{prop}
In order to recall Osgood Lemma we introduce first the notion of Osgood modulus of continuity, see \cite[Definition 3.1]{BCD} for more details.
\begin{defi}
Let $a>0$. A modulus of continuity $ \eta: [0,a] \to [0,+\infty[$ is any continuous non-decreasing function vanishing at $0$ and continuous at $0$.
We say that $\eta$ is an Osgood modulus of continuity if in addition 
$$\int_0^a \frac{dz}{\eta(z)} =  \infty.
$$
\end{defi}
We recall Osgood lemma which will be useful for the proof of Theorems \ref{thm0} and \ref{thm} in the critical case $p=3$.
\begin{lemme}[Osgood Lemma]\label{Osgood}
Let $f$ a measurable function from $[t_0,T]$ to $[0,a]$, $v$ a locally integrable function from $[t_0,T]$ to $\R^+$ and $ \eta: [0,a] \to [0,+\infty[$ an Osgood modulus of continuity. Assume that there exists $c \geq0$ such that for a.e $t\in[t_0,T]$
$$
f(t) \leq c+\int_{t_0}^t v(s)\, \eta(f(s)) ds.
$$
\begin{enumerate}
\item If $c>0$ then we have 
$$
\mathcal{M}(c) \leq \mathcal{M}(f(t)) + \int_{t_0}^t v(s) ds \text{ with } \mathcal{M}(x)=\int_x^a \frac{dz}{\eta(z)}.
$$
\item If $c=0$ then $f=0$ a.e on $[t_0,T]$.
\end{enumerate}
\end{lemme}
%

\section{Proof of Theorem \ref{thm} }
\label{sec-stabb}

Let $p\geq 3$, $R>0$,  a couple of initial densities 
  $ (\rho_1^0 , \rho_2^0 )$   in $E_p $ with  $ \max_{i=1,2} (\|\rho_i^0\|_{ L^1 (\R^3) },\|\rho_i^0\|_{ L^p (\R^3) }) \leq R$, 
and $  (u_1,\rho_1) $ and $ (u_2 ,\rho_2) $ in $C([0,+\infty),W^{2,p} \times  E_p)$ 
satisfying the transport-Stokes equation \eqref{TS} 
  for any $t\geq 0$, with, respectively,  $ \rho_1^0 $ and $\rho_2^0 $ as initial condition.

 Let us recall that 
  if $\rho_2$ is absolutely continuous with respect to the Lebesgue measure then 
\begin{equation}
    \label{W-infi}
 W_1 (\rho_1,\rho_2)=  \underset{T\# \rho_2=\rho_1} \inf \int_{\R^3} |T(x)-x|d\rho_2(x) ,
  \end{equation}
where the infimum runs over all measurable transport maps $T : \mathbb{R}^3 \to \mathbb{R}^3$ and there exists an optimal transport map $T^*$ for which the infimum is reached, see \cite[Theorem 1.5]{Santambrogio}. We emphasize that we use transport maps in the whole analysis for the sake of lightness and that one can adapt all the arguments using only optimal transport plans which are known to exist, see \cite[Theorem 1.7]{Santambrogio_book}.

We introduce $T_0$ the optimal transport map such that
\begin{equation}
    \label{enplu}
    \rho_1^0=T_0\#\rho_2^0 ,
    \end{equation}
and 
\begin{equation}
    \label{azero}
W_1(\rho_1^0,\rho_2^0)=\int_{\R^3} |T_0(x)-x| \rho_2^0(x) dx. 
\end{equation}
 For  $i=1,2$, let 
   $X_i$ associated  the characteristic flow with $u_i$ by 
\begin{equation} \label{cflow}
\left \{
\begin{array}{rcll}
\partial_t X_i(t,s,x) & = & u_i(s, X_i(t,s,x)),&\: \: \forall \, t,s\in[0,+\infty) , \\
X_i(s,s,x) & = & x,& \: \: \forall \, s \in [0,+\infty) ,
\end{array}
\right.
\end{equation}
For  $i=1,2$,
for any $t$ in $[0,+\infty)$, 
\begin{equation}
    \label{pousse}
    \rho_i (t,\cdot) =X_i (t,0,\cdot) \#\rho_i^0 .
    \end{equation}
We set, for any $t\geq 0$, 
\begin{equation}
    \label{RG}
T_t :=X_1(t,0,\cdot) \circ T_0 \circ X_2(0,t,\cdot)  .
\end{equation}
Let us define, for any $t\geq 0$, 
\begin{equation*}
 Q(t) :=  \int_{\R^3} |T_t(x) -x| \rho_2^t(dx)  .
\end{equation*}
By   \eqref{azero}, we observe that 
\begin{equation}
    \label{prem-zero}
     W_1 (\rho_1^0,\rho_2^0) =Q(0) .
\end{equation}
In what follows we use the shortcut $\rho_i^t:=\rho_i(t,\cdot)$ for $t\geq 0 $ and $i=1,2$.
From  \eqref{pousse},   \eqref{enplu} and \eqref{RG}  we deduce that for any $t\geq 0$, 
\begin{equation}
    \label{enplusse}
  \rho_1^t=T_t \#\rho_2^t   .
    \end{equation}
 We therefore deduce from  \eqref{W-infi} and  \eqref{enplusse}
 that for any $t\geq 0$, 
\begin{equation}
    \label{prem}
     W_1 (\rho_1^t,\rho_2^t) \leq  Q(t) . 
\end{equation}
Now, using \eqref{pousse} (with $i=2$) and \eqref{RG} (which implies that for any $t\geq 0$, 
 $T_t \circ X_2(t,0,\cdot)=X_1(t,0,\cdot) \circ T_0 $), 
 we have for any $t\geq 0$, 
\begin{equation}
    \label{reexpQ}
Q(t) 
= \int_{\R^3} | X_1(t,0,T_0 (x)) -X_2(t,0,x)| \rho_2^0(dx).
\end{equation}
From \eqref{cflow}
we deduce that for any $x$ in $\R^3$, for any $t\geq 0$, 
\begin{align} \nonumber
| X_1(t,0,T_0 (x)) -X_2(t,0,x)| & \leq |T_0 (x)- x| + \int_0^t |u_1(s,{X}_1(s,0,T_0 (x))) -u_2(s,{X}_2(s,0,x))| ds 
\\ \nonumber & \leq |T_0 (x)- x| + \int_0^t |u_1(s,{X}_1(s,0,T_0 (x))) -u_1(s,{X}_2(s,0,x))| ds
\\ \label{esti-arep} 
&  \quad + \int_0^t |u_1(s,{X}_2(s,0,x)) -u_2(s,{X}_2(s,0,x))| ds .
\end{align}
To go further we need to distinguish the case where $p> 3$ and the case where $p=3$. 
\subsection{Case where \texorpdfstring{$p> 3$}{p>3}}
Let us first deal  with  the case where  $p> 3$. 
Then by Proposition \ref{prop_log_lip_u}, the velocity field $u_1$ is Lipshitz and therefore for any $t\geq 0$, 
\begin{align*}
| X_1(t,0,T_0 (x)) -X_2(t,0,x)| 
& \leq |T_0 (x)- x| + \int_0^t \Lip(u_1(s)) |{X}_1(s,0,T_0 (x)) -{X}_2(s,0,x)| ds\\ \quad 
&+ \int_0^t |u_1(s,{X}_2(s,0,x)) -u_2(s,{X}_2(s,0,x))| ds .
\end{align*}
By integration over $\R^3$ with respect to the measure $\rho_2^0(dx)$, we deduce by 
 using \eqref{pousse} and \eqref{reexpQ} that for any $t\geq 0$, 
\begin{equation} \label{ineq-Q}
Q(t) \leq Q(0) + \int_0^t \Lip(u_1(s)) Q(s) ds+ \int_0^t \int  |u_1(s,x) -u_2(s,x)| \rho_2^s(dx) ds .
\end{equation}
 Let us now estimate the last term in the right hand side of \eqref{ineq-Q}.
By using the convolution formula  \eqref{formu-conv}
 and   \eqref{enplusse}
we obtain that for any $s\geq 0$, 
\begin{equation} \label{demen}
\int_{\mathbb{R}^3} |u_1(s,x)-u_2(s,x)| \rho^s_2(x)dx  
=  \int_{\mathbb{R}^3} \left| \int_{\mathbb{R}^3}( \mathcal{U}( x-T_s (y)) - \mathcal{U}( x-y) )e_3\rho_2^s(y)dy  \right| \rho_2^s(dx) .
\end{equation}
Thus we deduce from  \eqref{demen} and \eqref{Lip-prop} that for any $s\geq 0$, 
\begin{align*}
 &\int_{\mathbb{R}^3} |u_1(s,x)-u_2(s,x)| \rho^s_2(x)dx  
 \\ &\quad \quad \leq C \int_{\mathbb{R}^3} \left| \int_{\mathbb{R}^3}|T_s (y)-y| \left(\frac{1}{|x-y|^2}+ \frac{1}{|x-T_s (y)|^2} \right) \rho_2^s(y)dy  \right| \rho_2^s(dx),
 \end{align*}
 Therefore, by Fubini's principle, for any $s\geq 0$, 
  \begin{align} \label{triple}
  & \int_{\mathbb{R}^3} |u_1(s,x)-u_2(s,x)| \rho_2^s(x)dx  \leq  I_1 (s) Q(s) , 
\end{align}
  where
 \begin{equation*}
I_1 (s):= \sup_y \int_{\mathbb{R}^3} \left(\frac{1}{|x-y|^2}+ \frac{1}{|x-T_s (y)|^2} \right) )\rho_2^s (x)dx .
\end{equation*}
The integral corresponding to the first term in  $I_1 (s)$ can be estimated for all $y$ in $\R^3$ and for all $s\geq 0$ as follows:
\begin{align*}
 \int_{\mathbb{R}^3} \frac{1}{|x-y|^2}\rho_2^s(x)dx &\leq \int_{B(y,1)}\frac{1}{|x-y|^2} \rho_2^s(x) dx  + \int_{B(y,1)^c} \rho_2^s(dx) \\
 & \leq \left(\int_{B(y,1)} \frac{1}{|x-y|^{2p'}}\right)^{1/p'} \| \rho_2^s\|_{L^p(\R^3)} + \|\rho_2^s\|_{L^1(\R^3)} ,
 \end{align*}
 by H\"older's inequality, with $p'= p/(p-1)<3/2$, so that $z\mapsto z^{2-2p'}$ is integrable near the origin, and therefore
\begin{align*}
 \int_{\mathbb{R}^3} \frac{1}{|x-y|^2}\rho_2^s(x)dx 
 &\leq C\| \rho_2^s\|_{L^p(\R^3)} \int_0^1 z^{2-2p'} dz + \|\rho_2^s\|_{L^1(\R^3)} \leq C(\| \rho_2^s\|_{L^p(\R^3)}+ \|\rho_2^s\|_{L^1(\R^3)}).
\end{align*}
Proceeding similarly for the second term in $I_1 (s)$, we arrive at
\begin{equation}
    \label{borneI1}
    I_1 (s)  \leq C (\| \rho_2^s\|_{L^p(\R^3)}+ \|\rho_2^s\|_{L^1(\R^3)}),
\end{equation}
for any $s\geq 0$.
Combining \eqref{ineq-Q}, \eqref{triple} and \eqref{borneI1}  we get that for any $t\geq 0$, 
\begin{align*}
Q(t) \leq Q(0) + \int_0^t \Lip(u_1(s)) Q(s) ds+ \int_0^t C (\| \rho_2^s\|_{L^p(\R^3)}+ \|\rho_2^s\|_{L^1(\R^3)})Q(s) ds .
\end{align*}
By using $\|\rho_i^t\|_{L^q(\R^3)}=\|\rho_i^0\|_{L^q(\R^3)}$ for $q=1,p$ and $i=1,2$, \eqref{cotestokes} and a Gronwall-type argument, and combining this with  \eqref{prem} and \eqref{prem-zero} we arrive at \eqref{sta-geq-3}.
This concludes the proof of 
Theorem \ref{thm} in the case where  $p> 3$.

\subsection{Case where \texorpdfstring{$p= 3$}{p=3}}
Let us now deal  with  the critical case where  $p= 3$. 
By using \eqref{esti-arep} and Proposition \ref{prop_log_lip_u}, 
 we obtain that for any $x$ in $\R^3$, for any $t\geq 0$, 
 \begin{align*}
    & |X_1(t,0,T_0 (x))-X_2(t,0,x)| \leq |T_0 (x)-x|
    \\ & \quad +   C  \int_0^t |X_1(s,0,T_0 (x)) - X_2(s,0,x) |(1+\ln_{-}(|X_1(s,0,T_0 (x)) - X_2(s,0,x) |) ds\\
   &\quad +
     \int_0^t |u_1(s,X_2(s,0,x) - u_2(s,X_2(s,0,x) | ds ,
 \end{align*}
where $C=C(\|\rho_1^0\|_{L^3 (\R^3) \cap L^1 (\R^3)} )>0$. 
By integration over $\R^3$ with respect to the measure $\rho_2^0(dx)$, we deduce by 
 using \eqref{pousse} and \eqref{reexpQ} that for any $t\geq 0$, 
 \begin{gather*}
Q(t) \leq Q(0) 
\\  +  C \int_0^t\int_{\mathbb{R}^3} |X_1(s,0,T_0 (x)) - X_2(s,0,x) |(1+\ln_{-}(|X_1(s,0,T_0 (x)) - X_2(s,0,x) |) \rho_2^0(dx) ds
\\  +    \int_0^t \int_{\mathbb{R}^3}|u_1(s,X_2(s,0,x) - u_2(s,X_2(s,0,x) |\rho_2^0(dx) ds    .
 \end{gather*}
 Since  $z\mapsto z(1+\ln_{-}z)$ is concave for $z\geq 0$, it follows from Jensen's inequality 
 and \eqref{reexpQ} that for any $s\geq 0$, 
 \begin{align*}
\int_{\mathbb{R}^3} |X_1(s,0,T_0 (x)) - X_2(s,0,x) |(1+\ln_{-}(|X_1(s,0,T_0 (x)) - X_2(s,0,x) |) \rho_2^0(dx)
\\ \quad  \leq  Q(s)(1+\ln_{-}Q(s))
 .
 \end{align*}
 Therefore, for any $t\geq 0$, 
 \begin{align}\label{ineq-Q_p=3}
Q(t)
&\leq Q(0)+ C \int_0^t Q(s)(1+\ln_{-}Q(s)) ds+    \int_0^t \int_{\mathbb{R}^3}|u_1(s,x) - u_2(s,x) |\rho_2^s(x)dx ds .
 \end{align}
To bound the last term above, we establish the following result.
\begin{lemme}\label{LLint}
 For any $t\geq 0$, 
\begin{equation}
    \label{esti-LLint}
    \int|u_1(t,x)-u_2(t,x)|\rho_2(t,x)dx  \leq C(\|\rho_2^t\|_{L^1\cap L^3}) Q(t) (1+\ln_{-} Q(t)).
\end{equation}

\end{lemme}
\begin{proof}
By  \eqref{demen} and Fubini's principle, we obtain that 
for any $t\geq 0$, 
 \begin{align*}
&\int|u_1(t,x)-u_2(t,x)|\rho_2(t,x)dx \leq  \int \Big( \int |\mathcal{U}(x-T_t (y))-\mathcal{U}(x-y)| \rho_2(t,x)  dx \Big)  \rho_2^t(y) dy .
 \end{align*}
 Now we split the inner integral into three parts, corresponding to the domains 
 $B(T_t (y),\epsilon_y^t)$, $B(T_t (y),\epsilon_y^t)$ and $B(y,\epsilon_y^t)^c \cap B(T_t (y),\epsilon_y^t)^c$, where   $\epsilon_y^t =|T_t (y) - y|$.
 Observing that the Oseen tensor 
$\mathcal{U}$ satisfies that there exists $C>0$ such that for all $x $ in $\R^3 \setminus \{ 0\}$,
  $ \left|\mathcal{U}(x) \right|\leq C\frac{1}{|x|}  $, 
and using  \eqref{Lip-prop} we deduce from  \eqref{demen}   that for any $t\geq 0$,
 \begin{align*}
&\int|u_1(t,x)-u_2(t,x)|\rho_2(t,x)dx 
\leq C \int \int_{B(T_t (y),\epsilon_y^t)} \left(\frac{1}{|x-T_t (y) |}+\frac{1}{|x-y|} \right)\rho_2(t,x) dx \rho_2^t(y) dy\\
&\quad + C \int \int_{B(y,\epsilon_y^t)} \left(\frac{1}{|x-T_t (y)|}+\frac{1}{|x-y|} \right)\rho_2(t,x) dx \rho_2^t(y) dy\\
&\quad +C \int \epsilon_y^t \int_{B(y,\epsilon_y^t)^c \cap B(T_t (y),\epsilon_y^t)^c} \left(\frac{1}{|x-T_t (y)|^2}+\frac{1}{|x-y|^2} \right)\rho_2(t,x) dx \rho_2^t(y) dy\\
&=: I_1+I_2+I_3 .
 \end{align*}
 Regarding $I_1$ and $I_2$ we can proceed in the same way as in the proof of Proposition  \ref{prop_log_lip_u} and obtain 
 \begin{equation*}
 I_1 +I_2\leq C(\|\rho_2\|_{L^3 (\R^3)}) Q(t) .
 \end{equation*}
 For $I_3$ we distinguish two cases. First if $\epsilon_y^t=|y-T_t (y)|\geq 1 $ we get $$
\int_{B(y,\epsilon_y^t)^c \cap B(T_t (y),\epsilon_y^t)^c} \left(\frac{1}{|x-T_t (y)|^2}+\frac{1}{|x-y|^2} \right)\rho_2(t,x) dx \leq \|\rho_2^t\|_1.
$$
On the other hand, if $\epsilon_y^t=|y-T_t (y)|\leq 1 $, then 
\begin{align*}
&\int_{B(y,\epsilon_y^t)^c \cap B(T_t (y),\epsilon_y^t)^c} \left(\frac{1}{|x-T_t (y)|^2}+\frac{1}{|x-y|^2} \right)\rho_2(t,x) dx\\
&\leq \int_{B(T_t (y),\epsilon_y^t)^c } \frac{1}{|x-T_t (y)|^2}\rho_2(t,x)dx + \int_{B(y,\epsilon_y^t)^c } \frac{1}{|x-y|^2}\rho_2(t,x)dx\\
& \leq \int_{B(T_t (y),2)\setminus B(T_t (y),\epsilon_y^t) } \frac{1}{|x-T_t (y)|^2}\rho_2(t,x)dx + \int_{B(y,2)\setminus B(y,\epsilon_y^t) } \frac{1}{|x-y|^2}\rho_2(t,x)dx+ \|\rho_2\|_1\\
&\leq C(\|\rho_2\|_{L^1\cap L^3})(1+ |\log|T_t (y) -y||) .
\end{align*}
Gathering all the estimates we get 
$$
\int|u_1(t,x)-u_2(t,x)|\rho_2(t,x)dx  \leq C(\|\rho_2\|_{L^1\cap L^3}) \int |T_t (y) - y |(1+ \ln_{-}|T_t (y) - y |) \rho_2^t(y) dy.
$$

Using the fact that $z\mapsto z(1+\ln_{-} z)$ is concave for $z\geq 0$ we get
 \eqref{esti-LLint} and this concludes the proof of Lemma \ref{LLint}.
\end{proof}
We have for any $t\geq 0$, 
\begin{equation*}
Q(t) \leq  Q(0)+ C(\|\rho_i\|_{L^\infty(0,t, L^1\cap L^3)})\int_0^t Q(s)(1+\ln_{-} Q(s)) ds.
\end{equation*}
Thanks to the Osgood theorem, see Lemma \ref{Osgood} with $\eta(z)=z(1+\ln_{-}(z)) $, we deduce \eqref{sta-geq-critik}.
This concludes the proof of  Theorem \ref{thm}.

\section{Proof of Theorem \ref{thm0}}
\label{sec-proof-princi}

The proof of existence can be done by a density argument using the well known results for the existence of a solution in the case of regular data, see \cite{Hofer} and \cite{Mecherbet2}, however we propose below a self-contained proof based on a fixed-point argument.
 
 \subsection{Iterative scheme}
Let $p\geq 3$, $\rho_0 \in E_p$ and consider the sequence $(u^n,\rho^n)$ where 
$$ u^n(t) \in
\left\{
\begin{array}{lcl}
W^{2,p}(\R^3) & \text{ if } & p>3 ,\\
\dot{W}^{2,3}(\R^3)\cap \dot{W}^{1,3}(\R^3) \cap \underset{q\in(3,+\infty]}{\bigcap} W^{1,q}(\R^3) & \text{ if } & p=3 ,
\end{array}
\right.
$$
the unique solution to the Stokes equation
\begin{equation*}
\begin{array}{c}
-\Delta u^n+ \nabla p^n= - \rho^n e_3, \:\div(u^n)=0, \text{ on } \R^3 ,\\
\underset{|x|\to \infty}{\lim} |u^n| = 0 ,
\end{array}
\end{equation*}
and $\rho^{n+1}$ the unique solution to the transport equation
\begin{equation*}
\left\{
\begin{array}{rcl}
\partial_t \rho^{n+1}+\div(\rho^{n+1} u^n)&=& 0, \text{ on } \R^3 ,\\
\rho^{n+1}(0,\cdot)&=& \rho_0,
\end{array}\right.
\end{equation*}
with $\rho^0=\rho_0$. Since $u^n$ is divergence free and the associated flow $Xn$ is well defined, classical considerations for the transport equation ensure that $\rho^n \in  C([0,+\infty), L^1(\R^3)\cap L^p(\R^3)) $ for each $n$ and we have for any $t$,
\begin{equation}\label{borne_rho_n}
\|\rho^{n+1}(t)\|_{L^1(\R^3)}=\|\rho_0\|_{L^1(\R^3)}, \: \|\rho^{n+1}(t)\|_{L^p(\R^3)}=\|\rho_0\|_{L^p(\R^3)},
\end{equation}
Hence, we get the following uniform bounds of the velocity:
\begin{equation}\label{borne_u_n}
\|u^n(t)\|_{W^{2,p}(\R^3)}\leq C \|\rho^n(t)\|_{L^1(\R^3)\cap L^p(\R^3)} \leq C \| \rho_0\|_{L^1(\R^3)\cap L^p(\R^3)}
\end{equation}
in the case where  $p>3$
and 
\begin{equation}\label{borne_u_n_bis}
\LLip(u^n(t))+ \|\nabla u^n(t)\|_{W^{1,3}(\R^3)}+ \|  u^n(t)\|_{W^{1,q}(\R^3)}\leq C \|\rho^n(t)\|_{L^1(\R^3)\cap L^3(\R^3)} \leq C \| \rho_0\|_{L^1(\R^3)\cap L^3(\R^3)} ,
\end{equation}
in the case $p=3$ with $q\in(3,+\infty]$.
Moreover for $p\geq3$, we have, for all $0 \leq t \leq t' \leq T$,
\begin{align}\label{continuite_rho_n}
W_1(\rho^{n+1}(t),\rho^{n+1}(t')) &\leq \int_{\R^3}|X^n(t',0,x)-X^n(t,0,x)| \rho_0(dx) \notag\\
&\leq (t'-t) \|u^n\|_{L^\infty(0,T;L^\infty(\R^3))} \|\rho_0\|_{L^1(\R^3)} .
\end{align}
This shows that 
$$\rho^{n} \in  C([0,T],\mathcal{P}(\R^3)\cap L^1(\R^3)\cap L^p(\R^3))  \text{ uniformly in n for any }T>0.$$
Using the same arguments as in \eqref{demen}, We emphasize the following estimates for any compact set $K$ of $\R^3$
\begin{eqnarray}
    \|u_n(t)-u_n(t')\|_{W^{1,\infty}(K)}\leq C_K W_1(\rho^n(t),\rho^n(t')),& \text{ if } p>3,\label{equiconinuity_u_n(t),p>3}\\
        \|u_n(t)-u_n(t')\|_{L^{\infty}(K)}\leq C_K W_1(\rho^n(t),\rho^n(t')),& \text{ if } p=3 \label{equiconinuity_u_n(t),p=3} .
\end{eqnarray}

\subsection{Convergence of \texorpdfstring{$(\rho^n)_{n}$}{(rhon)n}}
 We introduce the following space for $M>0$
\begin{equation*}
B_{T,M}=\{\rho \in  C([0,T], \mathcal{P}(\R^3)), \quad
\underset{[0,T]}{\sup}W_1(\rho(t),\rho_0)  \leq M \} .
\end{equation*}
which is complete for the metric $d(\rho_1,\rho_2):=\underset{[0,T]}{\sup}W_1(\rho_1(t),\rho_2(t)) $.

Since $\rho^{n+1}(t):=X_n(t,0,\cdot)\#\rho_0 $ we have
\begin{align*}
W_1(\rho^n(t),\rho_0) \leq \int_{\R}|X^n(t,0,x)-x|\rho_0(dy) \leq T \|\rho_0\|_{L^1} \|u^n \|_{L^\infty(0,T;L^\infty)} ,
\end{align*}
which shows that $(\rho_n)_n$ lies in $B_{T,M}$ for an adequate choice of $M>0$. We aim to show that $(\rho_n)$ is a Cauchy sequence in $B_{T,M}$ for $T>0$ small enough. We present below the argument by distinguishing the case $p=3$ from the case $p>3$.

  \subsection*{Case where $p>3$.}
 For $p>3$, the stability estimates \eqref{ineq-Q} ensure that 
$$
W_1(\rho^n(t),\rho^{n+1}(t))\leq \left(\int_0^t \int_{\R^3}|u^{n-1}(s,x)-u^{n}(s,x)|\rho^{n}(s,x)dx ds \right)e^{M_n t} ,
$$
with $M_n=\left(\underset{0 \leq s \leq t }{\sup}\Lip(u^n(s))\right) $. On the other hand, performing an estimate as in \eqref{triple} yields
$$
\int_{\R^3}|u^{n-1}(s,x)-u^{n}(s,x)|\rho^{n}(s,x)dx\leq C(\|\rho^n\|_{L^1(\R^3)\cap L^p(\R^3)})W_1(\rho^{n-1}(s),\rho^n(s)) .
$$
Hence for all $t\in[0,T]$ we get using \eqref{borne_rho_n} and $M_n=\underset{0 \leq s \leq T }{\sup}\Lip(u^n(s))\leq M$ thanks to \eqref{borne_u_n}
\begin{equation*} 
\underset{t\in[0,T]}{\sup} W_1(\rho^n,\rho^{n+1})\leq C Te^{MT} \underset{t\in[0,T]}{\sup} W_1(\rho^n,\rho^{n-1}) .
\end{equation*}
This shows that for $T>0$ small enough, $\rho_n$ is a Cauchy sequence.

\subsection*{Case where $p=3$.}
For $p=3$ we use estimates \eqref{ineq-Q_p=3} to get for all $n,k\in \mathbb{N}$
\begin{multline}\label{eq:stab1_n,k}
W_1(\rho^{n+k+1}(t),\rho^{k+1}(t)) \leq C^{n,k} \int_0^t \eta \left(W_1(\rho^{n+k+1}(s),\rho^{k+1}(s)) \right)  ds\\
+ \int_0^t \int_{\R^3} |u^{n+k}(s,x)-u^{k}(s,x)|\rho^{n+k}(s,x) dx ds ,
\end{multline}
with $\eta(z)=z(1+\ln_{-}(z)) $ for $z\geq 0$ and $$C^{n,k}=\underset{s\in[0,T]}{\sup}\LLip(u^{n+k}(s)) \leq C \underset{s\in[0,T]}{\sup}\| \rho^{n+k}(s)\|_{L^1(\R^3)\cap L^3(\R^3)} \leq C ,$$ thanks to \eqref{esti-uli} and the uniform bounds on $(\rho^n)_{n}$.

On the one hand, using the uniform bounds of $(u_n)_n$ in $ C([0,T],L^\infty)$  and $(\rho_n)_n$ in $ C([0,T],L^{1})$ and the fact that $\eta(z)\leq \max(z,1) $ we have from $\eqref{eq:stab1_n,k} $
$$
\underset{0 \leq s \leq t}{\sup} W_1(\rho^{n+k+1}(t),\rho^{k+1}(t)) \leq C t \left(1+ \underset{0 \leq s \leq t}{\sup}W_1(\rho^{n+k+1}(t),\rho^{k+1}(t)) \right) + t C ,
$$
 which shows that, for $t$ small enough, the sequence $\left(\underset{0 \leq s \leq t}{\sup}W_1(\rho^{n+k}(t),\rho^{k}(t))\right)_{n,k}$ is uniformly bounded.

On the other hand, thanks to Lemma \ref{LLint} we have
\begin{equation}\label{eq:stab2_n,k}
\int_0^t \int_{\R^3} |u^{n+k}(s,x)-u^{k}(s,x)|\rho^{n+k}(s,x) dx ds\leq C \int_0^t  \eta \left(W_1(\rho^{n+k}(s),\rho^{k}(s)) \right)ds,
\end{equation}
where we used again \eqref{borne_rho_n}, \eqref{borne_u_n} to get a constant independent of $n,k$. Hence gathering inequalities \eqref{eq:stab1_n,k} and \eqref{eq:stab2_n,k} and setting $$f_{k}(t):=\underset{n}{\sup} \underset{0 \leq s \leq t}{\sup} W_1(\rho^{n+k}(s),\rho^{k}(s)) , $$ 
we get for all $T\geq 0$
$$
f_{k+1}(T) \leq C \int_0^T \eta(f_{k+1}(s))ds+ C\int_0^T \eta(f_k(s)) ds,
$$
where we used the fact that $z \mapsto \eta(z)$ is a non-decreasing function. By setting $\tilde{f}(t):=\underset{k\to \infty}{\limsup }f_k(t) $ we have 
$$
\tilde{f}(T) \leq 2 C \int_0^T \eta(\tilde{f}(s)) ds.
$$
We conclude thanks to Lemma \ref{Osgood} that $\tilde{f}=0 $ a.e on $[0,T]$, which shows that $ \rho^n$ is a Cauchy sequence.

\subsection*{Conclusion.}
Hence for $p\geq 3$, there exists a limit measure $\rho \in B_{T,M}\subset   C([0,T],\mathcal{P}(\R^3))$ such that $$\underset{t\in[0,T]}{\sup} W_1(\rho^n,\rho)\to 0.  $$

\subsection{Convergence of \texorpdfstring{$(u^n)_n$}{(un)n} and \texorpdfstring{$(X^n)_n$}{(Xn)n}} The second step is to extract a converging subsequence of $(u^n)_n$ by distinguishing again the case where $p>3$ from the case where $p=3$.

\subsection*{Case where $p>3$.}

Since for each compact set $K$ of $\R^3$ the sequence $ t \mapsto u_n(t)$ is equicontinuous in $W^{1,1}(K) $ thanks to \eqref{equiconinuity_u_n(t),p>3}, \eqref{continuite_rho_n} and $x\mapsto u_n(t,x)$ relatively compact in $ C^1(K) $  (using Ascoli with the equicontinuity of $x\mapsto u_n(t,x)$ in $ C^1(K)$), we get using again Ascoli the existence of $u\in C([0,T],  C^1(\R^3))$ such that $u_n \to u $ in $C([0,T],  C^1(K))$ up to a subsequence for each compact $K$.

Regarding the flow $X^n$, one can show that for each $0 \leq s\leq T $, $X^n(\cdot,s,\cdot) $ is uniformly bounded in $ C([0,T], C^{1}(K))$ on each compact set $K\in \R^3$, equicontinuous in $ C([0,T], C^{1}(\R^3))$ and that for each $t\in[0,T]$, $X^n(t,s,\cdot)$ is relatively compact in $ C^{1}(K)$. Indeed we have the following bounds for each $x,y \in \R^3 $, $t,t'\in[0,T]$, $|t-t'|<1$.
  \begin{gather}\label{borne_uniforme_x_t}
    |X_n(t,s,x)| \leq |x|+ |t-s| \|u^n\|_{ C(0,T; C(\R^3))},\: |\nabla X_n(t,s,x)|\leq e^{L(t-s)} ,
   \\ \label{equicontinuite_x}
|\nabla X^n(t,s,x)-\nabla X^n(t,s,y)| \leq |x-y|^\mu e^{L |t-s|} ,
\\ \label{equicontinuite_t}
    \|X^n(t,s,\cdot)-X^n(t',s,\cdot)\|_{ C^{1}(\R^3)} \leq Me^{LT}|t-t'|^\mu .
    \end{gather}
    From \eqref{equicontinuite_t} we get that $t\mapsto X^n(t,s,\cdot) $ is equicontinuous with values in $ C^{1,\mu}$,\eqref{borne_uniforme_x_t}  and  \eqref{equicontinuite_x} ensure that for each $t$, $X^n(t,s,\cdot) $  is relatively compact in $ C^1(K)$. Hence, since $X^n$ is uniformly bounded in $  C([0,T], C^{1,\mu}(K))$ thanks to \eqref{borne_uniforme_x_t}, then using the Ascoli theorem, there exists $X(\cdot,s,\cdot)\in  C([0,T], C^{1,\mu}(\R^3)) $ such that $X^n(\cdot,s,\cdot) \to X(\cdot,s,\cdot) $ in $  C([0,T], C^1(K)) $ for each compact $K$. In particular passing in the limit we have for all $s,t\in[0,T]$, 
    $X(t,s,\cdot) \circ X(s,t,\cdot)= Id$, 
    and using the convergence of $u^n$ to $u$ in $ C([0,T], C^1(K))$ for each compact $K$ we get that $X$ satisfies \eqref{cflow-def}.
    
    \subsection*{Case where $p=3$.}
  We proceed analogously using \eqref{equiconinuity_u_n(t),p=3} and show that there exists $u\in C([0,T], C(\R^3)) $ such that for each compact set $K$ we have $u^n\to u$ in $ C([0,T]; C(K))$.

Regarding the flow $X^n$,  
$X^n(\cdot,s,\cdot)-Id$ is uniformly bounded in $ C([0,T], C(\R^3))$ and in particular in $ C([0,T], C(K))$ for each compact set. Now, let $x,y \in \R^3$, $|x-y|\leq 1/2$ and $T^n$ the maximal time such that $|X^n(t,s,x)-X^n(t,s,y)|<3/4 $ 
for $t,s \in[0,T]$, we have 
\begin{align*}
|X^n(t,s,x)-X^n(t,s,y)|&\leq |x-y|+\left(\underset{\tau\in[0,T]}{\sup}\LLip(u^n(\tau)) \right) \\
&\times \int_s^t |X^n(t,s,x)-X^n(t,s,y)|(1+\ln_{-} |X^n(t,s,x)-X^n(t,s,y)|),\end{align*} 
where $\left(\underset{\tau\in[0,T]}{\sup}\LLip(u^n(\tau)) \right) $ is uniformly bounded thanks to \eqref{borne_u_n_bis}. Using Lemma \ref{Osgood} with $\eta(z):=z(1+\ln_{-}(z))$, we get for some constant $\bar C$ for any $t\leq T^n$, $|x-y|\leq1/2$
\begin{equation}\label{equicontinuite_x_3}
|X^n(t,s,x)-X^n(t,s,y)| \leq \bar{C}|x-y|^{e^{-CT}},
\end{equation}
which shows that the maximal time $T^n$ can be taken independent of $n$ and $T^n=T$ for $T$ small enough. hence, $X_n(t,s,\cdot)$ is relatively compact in $ C(K)$ for each $t,s \in [0,T]$. Moreover
$$
\|X^n(t,s,\cdot)-X^n(t',s,\cdot)\|_{ C(K)}\leq |t-t'| \| u^n\|_{ C(0,T;L^\infty(\R^3))},
$$
which ensures that $t\mapsto X^n(t,s,\cdot) $ is equicontinuous for each $s\in[0,T]$. Using again the Ascoli theorem this allows us to extract a converging subsequence in $ C([0,T], C(K))$. This allows us to construct $X\in  C([0,T], C(\R^3))$ and by passing in the limit in \eqref{equicontinuite_x_3} pointwisely we get for all $x,y \in \R^3$, for all  $t,s$ in $[0,T]$, 
$$
|X(t,s,x)-X(t,s,y)| \leq \bar{C}|x-y|^{e^{-CT}},
$$
and analogously to the case $p>3$ we have 
for $s,t\in[0,T]$, 
    $
    X(t,s,\cdot) \circ X(s,t,\cdot)= Id.
    $

\subsection{Existence and uniqueness of the solution}
Convergence of $\rho^n$ and $u^n$ (up to a subsequence) allows to pass in the limit in both the Stokes and transport equations and show that $(u,\rho)$ satisfies the transport-Stokes system weakly. Moreover, using the fact that $X^n $ satisfies
$$
\rho^{n+1}(t,\cdot)= X^n(t,0,\cdot)\#\rho_0,
$$
one can pass in the limit in both sides thanks to the strong convergence of $X^n$ up to a subsequence and the convergence of $\rho^n$  using the Wasserstein metric. Indeed, we get for all  $\psi\in  C_b(\R^3)$ 
$$
\int \psi(x) \rho(t,x) dx =\int \psi(X(0,t,x)) \rho_0(x) dx,
$$
which means
$$
\rho(t,\cdot)= X(t,0,\cdot)\# \rho_0,
$$
and thanks to the fact that  $X\in C([0,T], C(\R^3))$  we recover that $\rho \in  C([0,T], L^p(\R^3) \cap L^1(\R^3)) $. 
 From this it follows that $u \in  C([0,T], W^{2,p}(\R^3)) $ with the bounds \eqref{reg_u,p} (resp. $\nabla u \in  C([0,T], W^{1,3}(\R^3))$, $u\in  C([0,T],W^{1,q}(\R^3)) $ for $q\in]3,+\infty]$ with the bounds \eqref{reg_u,p_bis} ). 

Global in time existence is ensured thanks to the uniform bounds of $\rho(t)$ in $L^p(\R^3) \cap L^1(\R^3)$ while uniqueness is ensured thanks to Theorem \ref{thm}. This concludes the proof of  Theorem \ref{thm0}.

\section{Proof of Proposition \ref{apana}}\label{sec-apana}

This section is devoted to the proof of Proposition \ref{apana}.
In the course of the proof we shall use a technical lemma, see Lemma \ref{lemme_Miot_general}, which is proved in the next section. 

\subsection{A first few basic estimates}
We gather below some useful estimates. 
First, by \eqref{cflow-def}, for all $x$ and $y$ in $\R^3$, 
\begin{equation}
    \label{coince}
  |x-y|e^{-\int_0^t \Lip(u(s,\cdot))ds}\leq |X(t,x)-X(t,y)| \leq |x-y| e^{\int_0^t \Lip(u(s,\cdot))ds}.
\end{equation}
Thus, for $T>0$ small enough, 
this implies that there exists 
 $\lambda \in (1,3/2]$ such that
 \begin{equation}\label{aveclambda}
\frac{1}{\lambda} \leq  \frac{|X(t,x)-X(t,y)}{|x-y|} \leq \lambda,\: \forall \, x \neq y,\: t\in[0,T] .
 \end{equation}
Next, by splitting the integral into two parts, distinguishing the cases where $|x-z|\leq 1$ and where $|x-z|> 1$, and using H\"older's inequality for the first case with the observation that $p>3$, we obtain the existence of 
 $L=L(p)>0$ such that for all $0\leq s\leq 2$, 
\begin{equation}\label{fors}
\underset{x \in \R^3}{\sup} \int \frac{\rho_0(dz)}{|x-z|^s} \leq L \|\rho_0\|_{L^1\cap L^p} .
\end{equation}
Finally, from the definition of the the Oseen tensor in 
\eqref{Oseen} we deduce the following property.
\begin{lemme}\label{lemme_oseen_rec}
There exists 
$K>0$ such that for all $x$  in $\R^3\setminus \{0\}$ and all multi-index $\alpha \in \mathbb{N}^3$
\begin{equation}
    \label{alpha-U}
|\partial^\alpha \mathcal{U}(x)| \leq \frac{K^{|\alpha|} |\alpha|!}{|x|^{1+|\alpha|}} .
\end{equation}
\end{lemme} 
\begin{proof}
Let $n\geq 1$ and $\alpha \in \mathbb{N}^3$ a multi-index such that $|\alpha|=n $. We consider first the term $u(x):=\frac{1}{|x|}$ in the Oseen tensor. Since \begin{equation} \label{marre}
    \partial_{x_i} \frac{1}{|x|^p}=-p \frac{x_i}{|x|^{p+2}} ,
\end{equation} 
one can show by induction that for all $|\alpha|=n$ there exists polynomials $P_k^\alpha$, $k=\lceil\frac{n}{2}\rceil,\cdots, n$ such that
\begin{equation}\label{eq:reccurence}
\partial^\alpha u(x)=\underset{k=\lceil \frac{n}{2}\rceil }{\overset{n}{\sum}} \frac{P_k^\alpha(x)}{|x|^{1+2k}},
\end{equation}
where $P_k^\alpha$ is a a sum of monomials in $x_1,x_2,x_3$ such that each monomial is of degree 
$  -n+2k \geq 0 $ and $\lceil \frac{n}{2}\rceil$ denotes the ceiling function of $\frac{n}{2}$. 

Indeed the induction formula is satisfied for $n=1$ as a consequence 
of  \eqref{marre} with $p=1$. 
Now we assume that \eqref{eq:reccurence} holds true up to $n$, and we observe that,  for any $1 \leq i \leq 3$ and for any $\alpha  \in \mathbb{N}^3$ with $|\alpha|=n$,  by the induction assumption \eqref{eq:reccurence}, Leibniz' rule and  \eqref{marre}, 
\begin{align}\label{eq:reccurence_derivation}
    \partial_i \partial^\alpha u(x)&= - \underset{k=\lceil \frac{n}{2}\rceil}{\overset{n}{\sum}} (1+2k)\frac{x_i P_k^\alpha(x)}{|x|^{3+2k}}+\underset{k=\lceil\frac{n}{2}\rceil}{\overset{n}{\sum}} \frac{\partial_i P_k^\alpha(x)}{|x|^{1+2k}} \notag \\
    &= \frac{\partial_i P^\alpha_{\lceil\frac{n}{2}\rceil}(x)}{|x|^{1+2\lceil\frac{n}{2}\rceil}}
     + \underset{k=\lceil\frac{n}{2}\rceil+1}{\overset{n}{\sum}} \frac{\partial_i P^\alpha_k(x)-(2k-1)x_i P^\alpha_{k-1}(x)}{|x|^{1+2k}}
     - (1+2n)\frac{x_i P_n^\alpha(x)}{|x|^{3+2n}} .
\end{align}
Since $\lceil \frac{n}{2} \rceil+1= \lceil \frac{n+1}{2} \rceil$ for $n\geq0$, all the terms in \eqref{eq:reccurence_derivation} except the first one correspond to a term of the form $\frac{P_k^\beta}{|x|^{1+2k}} $ with $\beta=e_i+\alpha$ and $k\in \{\lceil \frac{n+1}{2} \rceil,\dots,n+1 \}$. For the first term we distinguish two cases 
\begin{itemize}
    \item If $n=2p$, the polynomial $P_{\lceil \frac{n}{2} \rceil}^\alpha$ is of degree $-n+ 2\lceil \frac{n}{2} \rceil=-n+2p=0  $ hence $\partial_i P_{\lceil \frac{n}{2} \rceil}^\alpha=0 $.
    \item If $n=2p+1$, then $\lceil \frac{n}{2} \rceil=p+1=\lceil \frac{2p+2}{2} \rceil=\lceil \frac{n+1}{2} \rceil$ and hence the polynomial $\partial_i P_{\lceil \frac{n}{2} \rceil}^\alpha$ corresponds to a term of the form $P_k^\beta$ with $k=\lceil \frac{n+1}{2} \rceil$.
\end{itemize}
Now if we set $D_n$ the { maximal number of monomials appearing in} \eqref{eq:reccurence} for all $|\alpha|=n$ we have using \eqref{eq:reccurence_derivation}
$$
D_{n+1} \leq D_n + 2 D_n = 3 D_n \leq 3^{n} D_1 .
$$
If we set $C_n$ the largest coefficient (in absolute value) appearing in \eqref{eq:reccurence} for all  $\alpha  \in \mathbb{N}^3$ with $|\alpha|=n$, we have, by using \eqref{eq:reccurence_derivation}, that 
$$
C_{n+1} \leq (1+2n) C_{n} \leq \underset{k=1}{\overset{n}{\prod}} (2k+1) C_1 .
$$
We observe that $C_1=1$, that $D_1=1$ and that 
$$
|\partial^\alpha \frac{1}{|x|} |\leq \frac{D_n C_n}{|x|^{1+n}}  \leq \frac{3^{n} 2^{3n+2}(n!)} {|x|^{1+n}},
$$
with $n=|\alpha|$. Above we used that $$(2n+1)!=\left(\underset{k=1}{\overset{n}{\prod}} (2k+1)\right)\left(\underset{k=1}{\overset{n}{\prod}} (2k)\right)=2^n n! \left(\underset{k=1}{\overset{n}{\prod}} (2k+1)\right),
$$
and 
$$
(2n+1)! = (2n+1) (2n)! \leq n(2n+1) 2^{n+1} (n!)^2 \leq 2^{3n+1} 2^{n+1} (n!)^2 .
$$
For the second term in the Oseen tensor $\frac{x_ix_j}{|x|^3} $ one can show analogously that the derivatives of $u(x)=\frac{1}{|x|^3} $ are of the form
$$
\partial^\alpha u(x)= \underset{k=\lceil \frac{n}{2}\rceil}{\overset{n}{\sum}}\frac{P_k^\alpha(x)}{|x|^{2k+3}},
$$
where   in $P_k^\alpha$ is a sum of monomials  of degree $ -n+2k  $.
Using the same arguments as above,
we get that for  $\alpha  \in \mathbb{N}^3$ 
with $n=|\alpha|$
$$
\left|\partial^\alpha \frac{1}{|x|^3}\right|\leq \frac{3^{n} 2^{3n+5} }{|x|^{3+n}} (n+1)! \leq \frac{3^{11n}}{|x|^{3+n}} n! .
$$
we conclude by using Leibniz formula 
\begin{align*}
\left|\partial^\alpha \frac{x_i x_j}{|x|^3} \right|&= \underset{\beta+\gamma=\alpha}{\sum} \begin{pmatrix} \beta \\ \alpha \end{pmatrix} \partial^\beta \left( \frac{1}{|x|^3}\right) \partial^{\gamma} (x_ix_j) \\
&\leq \underset{|\alpha|-2\leq |\beta| \leq |\alpha|}{\sum} \begin{pmatrix} \beta \\ \alpha \end{pmatrix}  \frac{3^{11|\beta|}}{|x|^{3+|\beta|}}|\beta|! 2|x|^{2-|\alpha|+|\beta|}\\
&\leq \frac{3^{13|\alpha|}}{|x|^{1+|\alpha|}}|\alpha|!.
\end{align*}
where we used $ \underset{0\leq |\beta| \leq |\alpha|}{\sum} \begin{pmatrix} \beta \\ \alpha \end{pmatrix}=2^{|\alpha|} $.
This concludes the proof of Lemma \ref{lemme_oseen_rec}.
\end{proof}

\subsection{The case where \texorpdfstring{$n=1$}{n=1}}
From \eqref{cflow-def}, we deduce that 
\begin{align*}
\|\partial_t X(t,\cdot)\|_\infty \leq \|u\|_\infty \leq C \|\rho_0\| \quad \text{ and }
\quad 
\|\nabla \partial_t X(t,\cdot)\|_\infty \leq \|\nabla u\|_\infty \|\nabla X \|_\infty \leq C \|\rho_0\| \lambda ,
\\ \quad \text{ and }
\quad 
|\nabla \partial_t X(t,x)-\nabla \partial_t X(t,y)| \leq [\nabla u]_{0,\mu}|X(t,x)-X(t,y)|^\mu \|\nabla X\|_\infty + 
|\nabla X(t,x)-\nabla X(t,y)| \|\nabla u\|_\infty .
\end{align*}
Therefore by the 
 Gronwall lemma, we arrive at 
$$
[\nabla \partial_t X]_{0,\mu} \leq C\|u\|_{2,p}\lambda^\mu \|\nabla X\|_\infty e^{\|\nabla u\|_\infty t}.
$$
This entails that the induction assumption is satisfied 
for  $n=1$, with some $C_0=C_0(\|\rho_0\|, \lambda, T,\mu)$ such that 
\begin{equation}
\|\partial_t X(t,\cdot)\|_{1,\mu} \leq \frac{C_0}{2}.
\end{equation}

\subsection{Beginning of the iteration}
Assume that \eqref{rec_formule} is satisfied up to $n \in \N^*$, We need to prove by induction the bound \eqref{rec_formule} for the $L^\infty$ estimate $\|\partial_t^{n+1} X(t,\cdot)\|_\infty$, the Lipschitz estimate $\|\nabla \partial_t^{n+1}  X(t,\cdot)\|_\infty$ and the H\"older estimate $[\nabla \partial_t^{n+1} X(t,\cdot)]_{0,\mu}$. Hence we split the proof into three parts. 

From \eqref{cflow-def} and \eqref{formu-conv}, we deduce the key formula 
\begin{equation}
    \label{mainformula}
\partial_t X(t,x)=\int \mathcal{U}_3(X(t,x)-X(t,y)) \rho_0(y) dy,
\end{equation}
where $$\mathcal{U}_3(x):=- \mathcal{U}(x)e_3 .$$

\subsection{\texorpdfstring{$L^\infty$}{L infty} Estimates}

By \eqref{mainformula} and 
 the Multivariate Fa\`a di Bruno formula \eqref{FDB_formula}, 
\begin{align*}
\partial_t^{n+1} X(t,x)&= \partial_t^n \partial_t X(t,x)\\
&=n! \underset{1\leq |\alpha|\leq n}{\sum} \int (\partial^\alpha \mathcal{U}_3(X(t,x)-X(t,z))\underset{s=1}{\overset{n}{\sum}}\underset{P_s(n,\alpha)}{\sum}\underset{j=1}{\overset{s}{\prod}}\frac{\partial_t^{l_j} (X(t,x)-X(t,z))^{k_j}}{(k_j!) (l_j!)^{|k_j|}}\rho_0(dz),
\end{align*}
Since the integers $l_j$ in the formula above satisfy $l_j\leq n$, we get by using the (Lipschitz part of the) induction hypothesis \eqref{rec_formule} and   \eqref{alpha-U},
\begin{align*}
\left|\partial_t^{n+1} X(t,x)\right|
&\leq n! \underset{1\leq |\alpha|\leq n}{\sum} \int_{\R^3} \frac{K^{|\alpha|}|\alpha|!}{|X(t,x)-X(t,z)|^{1+|\alpha|}}\\
&\quad \times \underset{s=1}{\overset{n}{\sum}}\underset{P_s(n,\alpha)}{\sum}\underset{j=1}{\overset{s}{\prod}}\frac{\left| (-1)^{l_j-1}(l_j)!C_0^{l_j}C_1^{l_j-1} \begin{pmatrix}1/2\\ l_j \end{pmatrix}\right|^{|k_j|} |x-z|^{|k_j|}}{(k_j!) (l_j!)^{|k_j|}}\rho_0(dz) ,\end{align*}
and therefore 
\begin{align*}\left|\partial_t^{n+1} X(t,x)\right|
&\leq n!(-1)^n C_0^n C_1^n \underset{1\leq |\alpha|\leq n}{\sum} K^{|\alpha|}|\alpha|!(-1)^{|\alpha|}C_1^{-|\alpha|}\\
&\quad  \times \left(\int_{\R^3}
 \frac{|x-z|^{|\alpha|}}{|X(t,x)-X(t,z)|^{1+|\alpha|}}\rho_0(dz)\right) \underset{s=1}{\overset{n}{\sum}}\underset{P_s(n,\alpha)}{\sum}\underset{j=1}{\overset{s}{\prod}} \frac{\begin{pmatrix}1/2\\ l_j \end{pmatrix}^{|k_j|}}{k_j!},\end{align*}
 by using the definition of $P_s(n,\alpha)$.
 Then, by \eqref{aveclambda} and  \eqref{fors}, we deduce that 
\begin{align*}
\left|\partial_t^{n+1} X(t,x)\right|
&\leq n!(-1)^n C_0^nC_1^n\lambda L\|\rho_0\|\underset{1\leq |\alpha|\leq n}{\sum}|\alpha|!(-1)^{|\alpha|} (K C_1^{-1} \lambda)^{|\alpha|}\underset{s=1}{\overset{n}{\sum}}\underset{P_s(n,\alpha)}{\sum}\underset{j=1}{\overset{s}{\prod}} \frac{\begin{pmatrix}1/2\\ l_j \end{pmatrix}^{|k_j|}}{k_j!} .
\end{align*}
Hence, considering $C_0>0$ and $C_1>0$ such that 
 $K C_1^{-1} \lambda<1 $ and $L\|\rho_0\|\lambda<C_0/2 $, and
 using that, according to  \cite[Lemma 3.3]{CVW}, there holds:
\begin{equation}\label{formule_magique1}
\underset{1\leq |\alpha|\leq n}{\sum}|\alpha|!(-1)^{|\alpha|} \underset{s=1}{\overset{n}{\sum}}\underset{P_s(n,\alpha)}{\sum}\underset{j=1}{\overset{s}{\prod}} \frac{\begin{pmatrix}1/2\\ l_j \end{pmatrix}^{|k_j|}}{k_j!}=2(n+1)\begin{pmatrix}1/2\\ n+1\end{pmatrix},
\end{equation}
 we get that 
$$
\left|\partial_t^{n+1} X(t,x)\right|\leq  (n+1)!(-1)^n C_0^{n+1} C_1^n  \begin{pmatrix}1/2\\ n+1\end{pmatrix} .
$$ 


\subsection{Lipschitz estimates}

First, from \eqref{mainformula} we deduce, by derivation in space, that 
\begin{align} \label{lip-e}
\partial_t \nabla X(t,x)= \int_{\R^3} \nabla \mathcal{U}_3(X(t,x)-X(t,z))\nabla X(t,x) \rho_0(dz).
\end{align}
Moreover, by applying $\partial_t^{n}$ 
and applying the Leibniz formula, we arrive at 
\begin{align*}
\partial_t^{n+1} \nabla X(t,x)&=\underset{k=0}{\overset{n}{\sum}} \begin{pmatrix}n \\k \end{pmatrix} \int_{\R^3} \partial_t^k \left( \nabla \mathcal{U}_3(X(t,x)-X(t,z))\right) \partial_t^{n-k}\nabla X(t,x) \rho_0(dz).
\end{align*}
A subtlety here is that it is necessary to distinguish between the case where  $0 \leq k<n$ and the case where $k=n$.

First, for $0 \leq k<n$, it follows from the induction hypothesis \eqref{rec_formule} that 
\begin{align}\label{formule_reccurence_gradX}
\left| \partial_t^{n-k}\nabla X(t,x) \right|&\leq (-1)^{n-k-1}(n-k)! \begin{pmatrix}1/2 \\ n-k \end{pmatrix} C_0^{n-k} C_1^{n-k-1}.
\end{align}
Second using for all $0 \leq k\leq n$, by the multivariate Fa\`a di Bruno formula \eqref{FDB_formula}, 
\begin{align*}
\partial_t^k \left( \nabla \mathcal{U}_3(X(t,x)-X(t,z))\right)&=k! \underset{1\leq |\alpha|\leq n}{\sum} \left(\partial^\alpha \nabla \mathcal{U}_3 \right)(X(t,x)-X(t,z))\\
&\quad \times \underset{s=1}{\overset{n}{\sum}}\underset{P_s(n,\alpha)}{\sum}\underset{j=1}{\overset{s}{\prod}}\frac{\partial_t^{l_j} (X(t,x)-X(t,z))^{k_j}}{(k_j!) (l_j!)^{|k_j|}} .
\end{align*}
Therefore, for $0 \leq k <n$, by  \eqref{alpha-U} and  the induction assumption 
 \eqref{rec_formule},  which is assumed to be  satisfied up to $n \in \N^*$, 
\begin{align*}
&\int_{\R^3} \left|\partial_t^k \left( \nabla \mathcal{U}_3(X(t,x)-X(t,z))\right) \right| \rho_0(dz) \leq k! \underset{1\leq |\alpha|\leq k}{\sum} \int_{\R^3} \frac{K^{|\alpha|+1}(|\alpha|+1)!}{|X(t,x)-X(t,z)|^{2+|\alpha|} }\\
&\times \underset{s=1}{\overset{k}{\sum}}\underset{P_s(n,\alpha)}{\sum}\underset{j=1}{\overset{s}{\prod}}\frac{\left| (-1)^{l_j-1}(l_j)!C_0^{l_j}C_1^{l_j-1} \begin{pmatrix}1/2\\ l_j \end{pmatrix}\right|^{|k_j|} |x-z|^{|k_j|}}{(k_j!) (l_j!)^{|k_j|}}\rho_0(dz)\\
& \leq (-1)^{k} k!C_0^kC_1^k\underset{1\leq |\alpha|\leq k}{\sum}      C_1^{-|\alpha|}K^{1+|\alpha|}(1+|\alpha|)\lambda^{2+|\alpha|} \left(\int_{\R^3}\frac{\rho_0(dz)}{|x-z|^2} \right) \\
&\quad \times (-1)^{|\alpha|}|\alpha|!\underset{s=1}{\overset{k}{\sum}}\underset{P_s(k,\alpha)}{\sum}\underset{j=1}{\overset{s}{\prod}}\frac{ \begin{pmatrix}1/2\\ l_j \end{pmatrix}^{|k_j|} }{(k_j!) },
\end{align*}
by \eqref{aveclambda}.
Therefore,
\begin{align*}
\int_{\R^3} \left|\partial_t^k \left( \nabla \mathcal{U}_3(X(t,x)-X(t,z))\right) \right| \rho_0(dz) 
& \leq (-1)^{k} k!C_0^kC_1^k \underset{1\leq |\alpha|\leq k}{\sum}   (2KL \|\rho_0\| \lambda^2)   (2 C_1^{-1}K\lambda)^{|\alpha|} \\
&\quad \times (-1)^{|\alpha|}|\alpha|!\underset{s=1}{\overset{k}{\sum}}\underset{P_s(k,\alpha)}{\sum}\underset{j=1}{\overset{s}{\prod}}\frac{ \begin{pmatrix}1/2\\ l_j \end{pmatrix}^{|k_j|} }{(k_j!) }\\
&\leq (-1)^{k} (k+1)!C_0^{k+1}C_1^{k+1}/4  \begin{pmatrix}1/2\\ k+1 \end{pmatrix},
\end{align*}
where we used \eqref{formule_magique1} and assumed that $2 C_1^{-1}K\lambda<1  $ and $2KL \|\rho_0\| \lambda^2<\frac{C_0}{2}\frac{C_1}{4} $.

We treat the case $n=k$ slightly differently by assuming that $2KL \|\rho_0\| \lambda^2<\frac{C_0}{2}\frac{1}{4}\frac{1}{\|\nabla X\|_\infty} $ instead.
Thus,
\begin{align*}
\int_{\R^3} \left|\partial_t^n \left( \nabla  \mathcal{U}_3(X(t,x)-X(t,z))\right) \right| \rho_0(dz)  
&\leq (-1)^{n} (n+1)!C_0^{n+1}C_1^{n}\frac{1}{4 \|\nabla X\|_\infty}  \begin{pmatrix}1/2\\ n+1 \end{pmatrix} .
\end{align*}

Gathering the estimates we get
\begin{align*}
&\left|\partial_t^{n+1} \nabla X(t,x)\right|\\
&\leq \underset{k=0}{\overset{n}{\sum}} \begin{pmatrix}n \\k \end{pmatrix} (-1)^{n-k-1}(n-k)!  \begin{pmatrix}1/2 \\n-k \end{pmatrix}C_0^{n-k}C_1^{n-k-1} (-1)^k(k+1)!C_0^{k+1}\frac{C_1^{k+1}}{4} \begin{pmatrix}1/2 \\k+1 \end{pmatrix} \\
&\leq n!  C_1^{n} C_0^{n+1}\underset{k=0}{\overset{n}{\sum}} (k+1) (-1)^{n-k-1}  \begin{pmatrix}1/2 \\n-k \end{pmatrix} (-1)^k\frac{1}{4} \begin{pmatrix}1/2 \\k+1 \end{pmatrix} \\
&\leq (n+1)!  C_1^{n} C_0^{n+1}(-1)^n\begin{pmatrix}1/2 \\n+1 \end{pmatrix},
\end{align*}
where we used \cite[Lemma 4.2]{CVW} for $r=n$ to get 
\begin{equation}\label{formule_magique2}
\underset{k=0}{\overset{n}{\sum}} (k+1) (-1)^k \begin{pmatrix}1/2 \\k+1 \end{pmatrix}  (-1)^{n-k-1}\begin{pmatrix}1/2 \\n-k \end{pmatrix} \leq 4(n+1)(-1)^n  \begin{pmatrix}1/2 \\n+1 \end{pmatrix}.
\end{equation}
Indeed we emphasize that  $(-1)^n  \begin{pmatrix}1/2 \\n+1 \end{pmatrix}\geq 0 $ which allows us to bound the sum on $m=0 \dots n$ and $r=n$ by the double sum on $m,r=0 \cdots n$  in \cite[Lemma 4.2]{CVW}.


\subsection{H\"older estimates}
By  \eqref{lip-e} and Leibniz' rule, 
\begin{align*}
&\partial_t^{n+1}(\nabla X(t,x)-\nabla X(t,y))=
\\
&= \partial_t^n \int_{\R^3} \left[\nabla \mathcal{U}_3(X(t,x)-X(t,z))-\nabla \mathcal{U}_3(X(t,y)-X(t,z)) \right]\nabla X(t,x) \rho_0(dz)\\
&\quad+ \partial_t^n \int_{\R^3} \left[\nabla \mathcal{U}_3(X(t,y)-X(t,z)) \right] \left( \nabla X(t,x)-\nabla X(t,y) \right) \rho_0(dz)\\
&=\underset{k=0}{\overset{n}{\sum}} \begin{pmatrix} n\\ k \end{pmatrix} (A_k+ B_k),
\end{align*}
with 
\begin{align*}
A_k &:= \int_{\R^3} \partial_t^k \left(\nabla \mathcal{U}_3(X(t,x)-X(t,z))-\nabla \mathcal{U}_3(X(t,y)-X(t,z)) \right)\rho_0(dz) \left( \partial_t^{n-k}\nabla X(t,x)\right) ,
\\ B_k &:=\int_{\R^3}  \partial_t^k \left[\nabla \mathcal{U}_3(X(t,y)-X(t,z)) \right] \rho_0(dz) \partial_t^{n-k}\left( \nabla X(t,x)-\nabla X(t,y) \right) .
\end{align*}

For $A_k$, by the Multivariate Fa\`a di Bruno formula \eqref{FDB_formula},
\begin{align*}
& \partial_t^k \left(\nabla \mathcal{U}_3(X(t,x)-X(t,z))-\nabla \mathcal{U}_3(X(t,y)-X(t,z)) \right)\\
&=k! \underset{1\leq |\alpha|\leq k}{\sum} \Big\{  \left(\partial^\alpha \nabla \mathcal{U}_3 \right)(X(t,x)-X(t,z)) \underset{s=1}{\overset{k}{\sum}}\underset{P_s(k,\alpha)}{\sum}\underset{j=1}{\overset{s}{\prod}}\frac{\partial_t^{l_j} (X(t,x)-X(t,z))^{k_j}}{(k_j!) (l_j!)^{|k_j|}}\\
&\quad - \left(\partial^\alpha \nabla \mathcal{U}_3 \right)(X(t,y)-X(t,z)) \underset{s=1}{\overset{k}{\sum}}\underset{P_s(k,\alpha)}{\sum}\underset{j=1}{\overset{s}{\prod}}\frac{\partial_t^{l_j} (X(t,y)-X(t,z))^{k_j}}{(k_j!) (l_j!)^{|k_j|}} \Big\},
\end{align*}
and therefore 
\begin{equation}
    \label{tard}
 \partial_t^k \left(\nabla \mathcal{U}_3(X(t,x)-X(t,z))-\nabla \mathcal{U}_3(X(t,y)-X(t,z)) \right)
=k! \underset{1\leq |\alpha|\leq k}{\sum} G_\alpha(t,x,z)-G_\alpha(t,y,z),
\end{equation}
with, for $1\leq |\alpha|\leq k$, 
$$
G_\alpha(t,x,z):=\left(\partial^\alpha \nabla \mathcal{U}_3 \right)(X(t,x)-X(t,z)) \underset{s=1}{\overset{k}{\sum}}\underset{P_s(k,\alpha)}{\sum}\underset{j=1}{\overset{s}{\prod}}\frac{\partial_t^{l_j} (X(t,x)-X(t,z))^{k_j}}{(k_j!) (l_j!)^{|k_j|}}.
$$
We shall use the following Lemma.
%
\begin{lemme}\label{lemme_Miot_general}
Let $k\geq 1$ and
$1\leq |\alpha|\leq k$. Then  there exists a constant $C_p\geq 2L$ independent of $k$ such that
\begin{equation}\label{Rapp}
    \int \left|G_\alpha(t,y,z) \right|\rho_0(dz) \leq L \|\rho_0\|_{L^1\cap L^p} \Lambda ,
\end{equation}
and
\begin{equation}\label{Rapp-diff}
\int \left|G_\alpha(t,x,z)-G_\alpha(t,y,z)\right| \rho_0(dz)   \leq  \Lambda |x-y|^\mu  C_p\|\rho_0\|_{L^1\cap L^p},
\end{equation}
with 
\begin{equation}
    \label{def-Lamb}
\Lambda:= \lambda^{2+|\alpha|} K^{1+|\alpha|}(2+|\alpha|)! \left(K\lambda+1 \right)\left( \underset{s=1}{\overset{k}{\sum}}\underset{P_s(k,\alpha)}{\sum}\underset{j=1}{\overset{s}{\prod}}\frac{\|\partial_t^{l_j}X(t,\cdot) \|_{1,\mu}^{|k_j|}}{(k_j!) (l_j!)^{|k_j|}}\right).
\end{equation}
\end{lemme}
\begin{proof}
First, by  \eqref{alpha-U}, 
\begin{eqnarray*}
|G_\alpha(t,x,z)|&\leq \frac{K^{2+|\alpha|}(2+|\alpha|)!}{|X(t,x)-X(t,z)|^{2+|\alpha|}}\left( \underset{s=1}{\overset{k}{\sum}}\underset{P_s(k,\alpha)}{\sum}\underset{j=1}{\overset{s}{\prod}}\frac{\|\partial_t^{l_j}X(t,\cdot) \|_{1,\mu}^{|k_j|}}{(k_j!) (l_j!)^{|k_j|}}\right)|x-z|^{|\alpha|} ,
\end{eqnarray*}
and therefore, by \eqref{aveclambda}, 
\begin{eqnarray}\label{borne_G}
|G_\alpha(t,x,z)| \leq \Lambda_1 \frac{1}{|x-z|^2} ,
\end{eqnarray}
with 
$$
\Lambda_1= \lambda^{2+|\alpha|}K^{2+|\alpha|}(2+|\alpha|)!\left( \underset{s=1}{\overset{k}{\sum}}\underset{P_s(k,\alpha)}{\sum}\underset{j=1}{\overset{s}{\prod}}\frac{\|\partial_t^{l_j}X(t,\cdot) \|_{1,\mu}^{|k_j|}}{(k_j!) (l_j!)^{|k_j|}}\right).
$$
Observing that, since $\lambda>1$, $ \Lambda_1 \leq \Lambda$, where we recall that $\Lambda$ is defined in  \eqref{def-Lamb}, we conclude the proof of 
\eqref{Rapp} using \eqref{fors}.

Next, to prove \eqref{Rapp-diff}, 
we first observe that, for any distinct $x$, $y$ and $z$ in $\R^3$, we have that 
\begin{align}\label{app1}
|G_\alpha(t,x,z)-G_\alpha(t,y,z)| \leq \underset{u\in S_{x,y}}{\sup}|\nabla_u G_\alpha(t,u,z)| |x-y|,
\end{align}
where $S_{x,y}$ is a $C^1$ path from $x$ to $y$ 
such  that for all $u \in S_{x,y}$, 
\begin{align}\label{app2}
 |u-z| \geq \min(|z-x|,|z-y|).
 \end{align}
We refer to  \cite[Lemma 3.15]{Hofer} for the proof of existence of such a path. 

Now, using \eqref{alpha-U} and \eqref{aveclambda}, 
we obtain that
\begin{align*}
|\partial_{u} G_\alpha(t,u,z)|&\leq \frac{K^{2+|\alpha|}(2+|\alpha|)!}{|X(t,u)-X(t,z)|^{3+|\alpha|}}|u-z|^{|\alpha|}\left( \underset{s=1}{\overset{k}{\sum}}\underset{P_s(k,\alpha)}{\sum}\underset{j=1}{\overset{s}{\prod}}\frac{\|\partial_t^{l_j}X(t,\cdot) \|_{1,\mu}^{|k_j|}}{(k_j!) (l_j!)^{|k_j|}}\right)\\
&+  \frac{K^{1+|\alpha|}(1+|\alpha|)!}{| X(t,u)-X(t,z)|^{2+|\alpha|} }|\alpha||u-z|^{|\alpha|-1}\left( \underset{s=1}{\overset{k}{\sum}}\underset{P_s(k,\alpha)}{\sum}\underset{j=1}{\overset{s}{\prod}}\frac{\|\partial_t^{l_j}X(t,\cdot) \|_{1,\mu}^{|k_j|}}{(k_j!) (l_j!)^{|k_j|}}\right)\\
&\leq \lambda^{2+|\alpha|} K^{1+|\alpha|}(2+|\alpha|)! \left(K\lambda+1 \right) \frac{1}{|u-z|^3}\left( \underset{s=1}{\overset{k}{\sum}}\underset{P_s(k,\alpha)}{\sum}\underset{j=1}{\overset{s}{\prod}}\frac{\|\partial_t^{l_j}X(t,\cdot) \|_{1,\mu}^{|k_j|}}{(k_j!) (l_j!)^{|k_j|}}\right).
\end{align*}
By combining with \eqref{app1} and \eqref{app2}, 
this ensures that 
\begin{equation}\label{lipschitz_type}
|G_\alpha(t,x,z)-G_\alpha(t,y,z)|\leq \Lambda |x-y| \left(\frac{1}{|x-z|^3}+\frac{1}{|y-z|^3} \right),
\end{equation}
recalling the definition of 
$
\Lambda
$ in  \eqref{def-Lamb}.

Now, following  the proof of \cite[Lemma 2]{Miot} with the kernel $G_\alpha(t,x,z)$, we 
 set $d:=|x-y|$ and observe that, thanks to \eqref{borne_G}, it is enough to show the result for $d<1 $. Indeed if $d=|x-y|\geq 1$ then 
 \begin{equation}\label{eq:Holder_d>1}
 \int_{\R^3} |G_\alpha(t,x,z)-G_\alpha(t,y,z)|\rho_0(dz)\leq 2 \underset{x}{\sup}\int_{\R^3} |G_\alpha(t,x,z)|\rho_0(dz) \leq 2\Lambda_1 L \|\rho_0\|_{L^1\cap L^p} \leq 2\Lambda_1 L \|\rho_0\|_{L^1\cap L^p} |x-y|^\mu.
 \end{equation}
 Now assume $d<1$, we
 set $A:=\frac{x+y}{2}$, and we split  the following integral into three parts:
\begin{align}\label{trich}
&\int_{\R^3} |G_\alpha(t,x,z)-G_\alpha(t,y,z)|\rho_0(dz) =J_1 +J_2+ J_3,
\end{align}
where
\begin{align*}
J_1 :=&\int_{\R^3\setminus B(A,1)}|G_\alpha(t,x,z)-G_\alpha(t,y,z)|\rho_0(dz) ,
\\  J_2 :=&\int_{B(A,1)\setminus B(1,d)}|G_\alpha(t,x,z)-G_\alpha(t,y,z)|\rho_0(dz) ,
\\  J_3 :=& \int_{B(A,d)}|G_\alpha(t,x,z)-G_\alpha(t,y,z)|\rho_0(dz) .
\end{align*}
\begin{itemize}
    \item For $J_1$, we have for all $z\in B(A,d)$, $|z-x|\leq |z-1|+d/2 \leq 3d/2$. This yields,  by \eqref{borne_G}, that 
\begin{align*}
J_1&\leq 2 \Lambda_1 \int_{B(x,3d/2)} \frac{1}{|x-z|^2}\rho_0(dz)\\
&\leq  C\Lambda_1  \left(\int_0^{3d/2} s^{2-2p'} \right)^{1/{p'}} \|\rho_0\|_{L^p} ,
\end{align*}
and therefore 
\begin{align}\label{J1}
J_1 &\leq \tilde{C}_p \Lambda_1 \|\rho_0\|_{L^p} d^{3/{p'}-2}\leq  \tilde{C}_p \Lambda \|\rho_0\|_{L^p} d^{3/{p'}-2}.
\end{align}
\item For $J_2$ we use that for $ z\in B(A,1)\setminus B(A,d)$, 
$$
|z-x|>|z-A|-d/2>d/2
\text{ and } |z-x|\leq |z-A|+d/2 \leq 1+d/2 ,$$ 
and \eqref{lipschitz_type} to arrive at
\begin{align*}
J_2&\leq  |x-y|\Lambda \int_{ B(A,1)\setminus B(1,d)}\left(\frac{1 }{|x-z|^3}+\frac{1}{|y-z|^3}  \right)\rho_0(dz)\\
&\leq C |x-y|\Lambda \|\rho_0\|_{L^p} \left(\int_{d/2}^{1+d/2} s^{2-3p'}\right)^{1/{p'}} ds ,
\end{align*}
and therefore 
\begin{align}\label{J2}
J_2&\leq \tilde{C}_p |x-y|\Lambda \|\rho_0\|_{L^p}d^{3/{p'}-3}=\tilde{C}_p \Lambda \|\rho_0\|_{L^p}d^{3/{p'}-2} .
\end{align}
\item For $J_3$ we have that any $z$ such that $|z-A|>1$ satisfies 
$$
|z-x|>|z-A|- \frac{1}{2} d>|z-A|- \frac{1}{2} > \frac{1}{2} |z-A|,
$$
which yields using \eqref{lipschitz_type},
\begin{align}\label{J3}
J_3 &\leq \Lambda |x-y| \int_{\R^3\setminus B(A,1)}\left(\frac{1}{|x-z|^3}+\frac{1}{|y-z|^3}  \right)\rho_0(dz)\leq  16|x-y|\Lambda \|\rho_0\|_{L^1} .
\end{align}
\end{itemize}
By combining \eqref{eq:Holder_d>1},  \eqref{trich},  \eqref{J1},  \eqref{J2} and  \eqref{J3}, we establish the inequality \eqref{Rapp-diff} by choosing $C_p\geq \max(2L, 2\tilde{C}_p+16)$. 
This concludes the proof of Lemma \ref{lemme_Miot_general}.
\end{proof}
We start with the treatment of the term $A_k$ which is the most difficult.
 By combining \eqref{tard} and \eqref{Rapp-diff}
 we first deduce that
\begin{align*}
&\frac{1}{|x-y|^\mu}\int_{\R^3}
\left|  \partial_t^k \left(\nabla \mathcal{U}_3(X(t,x)-X(t,z))-\nabla \mathcal{U}_3(X(t,y)-X(t,z)) \right)\rho_0(dz)  \right| \\
& \leq k! \underset{1 \leq |\alpha|\leq k}{\sum} \lambda^{2+|\alpha|} K^{1+|\alpha|}C_p \|\rho_0\|_{L^1\cap L^p}(2+|\alpha|)! \left(K\lambda+1 \right)\left( \underset{s=1}{\overset{k}{\sum}}\underset{P_s(k,\alpha)}{\sum}\underset{j=1}{\overset{s}{\prod}}\frac{\|\partial_t^{l_j}X(t,\cdot) \|_{1,\mu}^{|k_j|}}{(k_j!) (l_j!)^{|k_j|}}\right) .
\end{align*}
 Then we use the induction assumption 
 \eqref{rec_formule},  which is assumed to be  satisfied up to $n \in \N^*$, to infer that 
\begin{align*}
&\frac{1}{|x-y|^\mu}\int_{\R^3}  \left|    \partial_t^k \left(\nabla \mathcal{U}_3(X(t,x)-X(t,z))-\nabla \mathcal{U}_3(X(t,y)-X(t,z)) \right)\rho_0(dz) \right| \\
& \leq k! \underset{1 \leq |\alpha|\leq k}{\sum} \lambda^{2+|\alpha|} K^{1+|\alpha|}C_p \|\rho_0\|_{L^1\cap L^p}(2+|\alpha|)! \left(K\lambda+1 \right)\left( \underset{s=1}{\overset{k}{\sum}}\underset{P_s(k,\alpha)}{\sum}\underset{j=1}{\overset{s}{\prod}}\frac{\left( (-1)^{l_j-1} (l_j!) \begin{pmatrix}1/2\\ l_j \end{pmatrix}C_0^{l_j}C_1^{l_j-1}\right)^{|k_j|}}{(k_j!) (l_j!)^{|k_j|}}\right)\\
& \leq k!C_0^k C_1^k  (-1)^k \underset{1 \leq |\alpha|\leq k}{\sum} \lambda^{2+|\alpha|} C_1^{-|\alpha|}(-1)^{|\alpha|}K^{1+|\alpha|}C_p \|\rho_0\|_{L^1\cap L^p}(2+|\alpha|)! \left(K\lambda+1 \right)\left( \underset{s=1}{\overset{k}{\sum}}\underset{P_s(k,\alpha)}{\sum}\underset{j=1}{\overset{s}{\prod}}\frac{ \begin{pmatrix}1/2\\ l_j \end{pmatrix})^{|k_j|}}{(k_j!)}\right)\\ 
&= k!C_0^k C_1^k  (-1)^k \underset{1 \leq |\alpha|\leq k}{\sum} \lambda^{2+|\alpha|} C_1^{-|\alpha|}K^{1+|\alpha|}C_p \|\rho_0\|_{L^1\cap L^p}2^{2+|\alpha|}2^{1+|\alpha|} \left(K\lambda+1 \right)\left( (-1)^{|\alpha|} |\alpha|! \underset{s=1}{\overset{k}{\sum}}\underset{P_s(k,\alpha)}{\sum}\underset{j=1}{\overset{s}{\prod}}\frac{ \begin{pmatrix}1/2\\ l_j \end{pmatrix})^{|k_j|}}{(k_j!)}\right).
\end{align*}

 We now assume that $4\lambda C_1^{-1}K<1$ and $ \lambda^2 K C_p \|\rho_0\|_{L^1\cap L^p}2^3(K\lambda+1)\leq  \frac{C_0}{2} \frac{C_1}{4}$, in order to get for $n\neq k$
\begin{align*}
&\frac{1}{|x-y|^\mu}
 \left|    \int \partial_t^k \left(\nabla \mathcal{U}_3(X(t,x)-X(t,z))-\nabla \mathcal{U}_3(X(t,y)-X(t,z)) \right)\rho_0(dz)  \right| \\
&\leq \frac{1}{4}(k+1)!C_0^{k+1} C_1^{k+1}  (-1)^k \begin{pmatrix}1/2\\ k+1 \end{pmatrix},
\end{align*}
by using again \eqref{formule_magique1}.
Again, we treat the case $n=k$ differently by assuming  $$ \lambda^2 K C_p \|\rho_0\|_{L^1\cap L^p}2^3(K\lambda+1)\leq  \frac{C_0}{2} \frac{1}{4\|\nabla X\|_{0,\mu}},$$ which yields
\begin{align*}
&\frac{1}{|x-y|^\mu}\int_{\R^3} \partial_t^n \left(\nabla \mathcal{U}_3(X(t,x)-X(t,z))-\nabla \mathcal{U}_3(X(t,y)-X(t,z)) \right)\rho_0(dz)\\
&\leq \frac{1}{4 \|\nabla X\|_{0,\mu}}(n+1)!C_0^{n+1} C_1^{n}  (-1)^n \begin{pmatrix}1/2\\ n+1 \end{pmatrix}.
\end{align*}

We get, by using again \eqref{formule_magique1} and \eqref{formule_reccurence_gradX}, that 
\begin{align*}
\underset{k=0}{\overset{n}{\sum}}\begin{pmatrix}n \\k \end{pmatrix}\frac{|A_k|}{|x-y|^\mu}&\leq 
\underset{k=0}{\overset{n}{\sum}}\begin{pmatrix}n \\k \end{pmatrix}(k+1)!C_0^{k+1} C_1^{k+1} (-1)^k   \frac{1}{4} \begin{pmatrix}1/2\\ k+1 \end{pmatrix}(-1)^{n-k-1}(n-k)! \begin{pmatrix}1/2 \\ n-k \end{pmatrix} C_0^{n-k} C_1^{n-k-1}\\
&\leq n! \underset{k=0}{\overset{n}{\sum}}\frac{1}{4} C_0^{n+1} C_1^{n}   (k+1)(-1)^k\begin{pmatrix}1/2\\ k+1 \end{pmatrix} (-1)^{n-k-1}\begin{pmatrix}1/2 \\ n-k \end{pmatrix} \\
&\leq C_0^{n+1} C_1^{n}  (-1)^n (n+1)!\begin{pmatrix}1/2 \\ n+1 \end{pmatrix},
\end{align*}
where we used \eqref{formule_magique2}.

Analogously for $B_k$ we emphasize that 
$$
\partial_t^k \left[\nabla \mathcal{U}_3(X(t,y)-X(t,z)) \right]=k! \underset{1\leq |\alpha|\leq k}{\sum} G_\alpha(t,y,z), 
$$
and that, analogously to \eqref{formule_reccurence_gradX} we have for $n \neq k$
$$
 \left| \partial_t^{n-k}\left(\nabla X(t,x)-\nabla X(t,y)\right) \right|\leq (-1)^{n-k-1}(n-k)! \begin{pmatrix}1/2 \\ n-k \end{pmatrix} C_0^{n-k} C_1^{n-k-1} |x-y|^\mu .
$$
Using \eqref{Rapp} and 
the two estimates above, we can treat the term $B_k$ as the term $A_k$ and conclude the proof.

\section{Proof of Theorem \ref{TS-control}}\label{sec-cont}
\label{sec-control}
This section is devoted to the proof of Theorem \ref{TS-control}. 
By  a reversibility argument, it is enough to find a control which sends the initial density $\rho_0$ to $0$ on the time interval $[0,T/2]$, indeed one can then proceed in a similar way and send $0$ to the final density $\rho_f$ in the time interval $[T/2,T]$ by switching the time variable $t$ into $T-t$.
We emphasize that it is actually sufficient to prove that it is possible to prove that there exists $\eps\in (0,\frac12)$ and a control which steers the initial density $\rho_0$ to $0$ on the time interval $[0,\eps T]$, since then one may turn off the control and let the density stays at zero during the rest of the interval $]\epsilon T, T/2[$.

 Without loss of generality we assume that there exists $\delta>0$ small enough such that $B(0,4\delta) \subset \omega$.

 \subsection{Construction of a controlled auxiliary solution  starting from $0$ and returning at $0$}

 Following \cite{G20}, in this subsection we construct, for any $T>0$, a vector field $u$ with zero initial and final values and such that during the imparted time the corresponding flow map sends the support of $\rho_0$ to the ball $B(0,2 \delta) $ in several time steps and then back to the support of $\rho_0$. More precisely we have the following result.

 \begin{prop}\label{lemme1}
Let $T>0$, $\omega$ and $\rho_0$ as above. There exists  $f\in C_c((0,T);L^\infty(\R^3))$ compactly supported in $(0,T)\times \omega$  and $(u_\aux,p_{\aux}) \in C_c((0,T);\bigcup_{q\geq3} W^{2,q}(\R^3) \times \bigcup_{q\geq3}W^{1,q}(\R^3))$ satisfying 
$$-\Delta u_\aux +\nabla p_\aux = f \quad  \text{ and } \quad \div u_\aux=0,
 \quad \text{ on } \quad \R^3 , $$ 
 and
there exists  a covering $\underset{i=1}{\overset{L}{\bigcup}} B(x_i,r_i)$ of $\supp \rho_0$ and a sequence $0\leq t_i<t_{i+1/4}<t_{i+1/2}<t_{i+1}\leq T/2$, $0\leq i \leq L $ with $t_0=T/4$, $t_{L+1}=T/2$ such that 
\begin{equation}\label{forumule_controle_zone}
 \forall 1 \leq i \leq L,\: \forall x\in B(x_i, r_i), \quad \forall \, t \in ]t_{i+1/4},t_{i+1/2}[ \,\Rightarrow
X_\aux(t,0,x)\in B(0,2\delta) ,
\end{equation}
$$
\forall 1 \leq i \leq L, \forall\, x \in \R^3, X_{\aux}(t_i,0,x)=x .
$$
\end{prop}

\begin{proof}[Proof of Proposition \ref{lemme1}]
By a scaling argument we may choose $T=1$. Let $\alpha \in \omega $ and $r>0$ such that $B(\alpha,r)\subset \omega \setminus \overline{}{B(0,4 \delta)}$.
The proof relies on the following Lemma that we prove later.
\begin{lemme}\label{lemme2}
Let $ \gamma : [0,1] \to \R^3\setminus \overline{B(\alpha,r)}$ a  curve of class $C^\infty$ such that 
$\gamma(t)=\gamma(0)$ for $t \in [0,\zeta]$ and 
$\gamma(t)=\gamma(1)$, for $ t \in [1-\zeta,1]$, 
for a given small $\zeta>0$. Then there exists
$f\in C(0,1,L^\infty(\R^3))$ compactly supported in $(0,1)\times B(\alpha,r)$ and the unique associated solution $u\in C_c(0,1,\bigcap_{q\geq 3}W^{2,q}(\R^3) )$ of 
$$-\Delta u +\nabla p = f,\quad \div u =0,\quad  \text{ on } \R^3$$
such that the associated characteristic flow $X$  satisfies $X(t,0,\gamma(0))=\gamma(t)$.
\end{lemme}

Now let $a \in \supp \rho_0 \setminus \overline{B(\alpha,r)}$. 
There exists  a  curve $\gamma_a:[0,1] \to \R^3 \setminus \overline{B(\alpha,r)}$  of class $C^\infty$ such that $$\gamma_a(t)=a \text{ for } t\in[0,1/4], \: \gamma_a(t)=0 \text{ for }t\in[1/2, 1]$$ 
Denote by $u_a$ the velocity and $f_a$ the right hand side obtained by Lemma \ref{lemme2} associated to the curve $\gamma_a$. 
By continuity of the flow $X_a$, there exists $r_a>0$ and $\eta_a>0$ small enough such that for all $x \in B(a,r_a)$ and for all $t \in ]1/2-\eta_a,1/2+\eta_a[$ we have $X_a(t,0,x)\in B(0,2\delta)$. We emphasize that the stationarity of the curve near $t=0$ and $t=1$ is crucial in order to have a compactly supported velocity field $u_a$.

Now for $ a \in \supp \rho_0\cap \overline{B(\alpha,r)} $ we consider  $B(\beta,r')$ such that  $B(\beta,r') \subset  \omega \setminus \overline{\left( B(0,4\delta)\cup B(\alpha,r)\right )}$ and apply Lemma \ref{lemme2} for $B(\beta,r')$ instead of $B(\alpha,r)$ in order to drive each point $a$ of $\supp \rho_0\cap \overline{B(\alpha,r)}$ to $0$. This yields the existence of a control $u_a$ satisfying a Stokes equation with a right hand side $f_a$ compactly supported in $B(\beta,r')$ and we may apply the same continuity argument as above to get the existence of $r_a$ and $\eta_a$ such that $X_a(t,0,x)\in B(0,2\delta)$ for $x\in B(a,r_a)$, $t\in ]1/2-\eta_a,1/2+\eta_a[$.

Hence, to summarize, for all $a\in \supp \rho_0 $, there exists $r_a>0$ and $u_a$ compactly supported in $(0,1)$  which solves the Stokes equation with a right hand side $f_a$ compactly supported in $(0,1)\times \omega$ and such that for all $x \in B(a,r_a)$, in $X_a(t,0,x) \in B(0,2\delta)$ for $t\in]1/4,1/2[$.

By compactness one can extract a finite subcover
$$\underset{i=1}{\overset{L}{\bigcup}}B(x_i,r_i),$$ 
of $\supp \rho_0$. 
Let $(u_i)_i$ the associated velocity fields.
 We split the time segment $[1/4,1/2]$  by considering the subdivision%
\begin{gather*}
t_0=1/4, \\
t_i=t_0+ \frac{i}{4(L+1)}, \quad \forall\: i \in \{1, \cdots, L+1\},\\
t_{i+1/4}=t_i +\frac{1}{4} \frac{1}{4(L+1)}, \quad \forall\: i \in \{0, \cdots, L\}, \\ 
t_{i+1/2}=t_i +\frac{1}{2} \frac{1}{4(L+1)},\quad  \forall\: i \in \{0, \cdots, L\}.
\end{gather*}
 We set then 
\begin{align*}
    u_\aux(t,x)&=8(L+1)u_i(8(L+1)(t-t_{i}),x), \: 0 \leq i \leq L, \: t \in [t_{i},t_{i+1/2}],
\\ u_\aux(t,x)&=-8(L+1)u_i(8(L+1)(t_{i}-t),x), \: 0 \leq i \leq L, \: t \in [t_{i+1/2},t_{i+1}].
\end{align*}
$u_\aux$ is then well defined and compactly supported in $(0,T)$ since each $u_i$ is compactly supported. Moreover, $u_\aux$ solves the Stokes equation on $\R^3 \setminus \omega$.
Moreover, on each time segment $]t_{i},t_{i+1}[ $  the domain $B(x_i,r_i)$ is transported inside $B(0,2\delta)$ during the time interval $]t_{i+1/4},t_{i+1/2}[$ and then again to $B(x_i,r_i)$ during the time interval $]t_{i+1/2},t_{i+1}[$, more precisely 
 the characteristic flow  $X_\aux$ associated with $u_\aux$ satisfies
\begin{equation*}
 \forall 1 \leq i \leq L,\: \forall x\in B(x_i, r_i), \quad \forall \, t \in ]t_{i+1/4},t_{i+1/2}[ \,\Rightarrow
X_\aux(t,0,x)\in B(0,2\delta).
\end{equation*}
$$
 \forall 1 \leq i \leq L,\forall \, x \in \R^3, \: X_\aux(t_{i},0,x)=x .$$
\end{proof}

Following \cite[Proposition 3.2.8]{G}, the proof of Lemma \ref{lemme2} relies on the following result which establishes the pointwise reachability of any value in $\R^3$ by the velocity vector fields satisfying  controlled steady Stokes systems. 
\begin{lemme}\label{lemme3}
Let $ x \in \R^3 \setminus B(\alpha,r)$, then we have
$$
\{ u(x),\exists \, f\in L^\infty(\R^3) \text{ compactly supported in } B(\alpha,r) \text{st } -\Delta u+\nabla p =f  \text{ and }   \div u= 0  \text{ on } \R^3\}=\R^3.
$$
\end{lemme}
\begin{proof}[Proof of Lemma \ref{lemme3}]
 We set for a given $x \in \R^3 \setminus B(\alpha,r)$ and any $a \in \R^3$
$$
F:=8\pi|x-\alpha | \left(I-\frac{1}{2}\frac{(x-\alpha) \otimes (x-\alpha)}{|x-\alpha|^2} \right)a \in \R^3 ,
$$
we consider then the velocity given by $u=\mathcal{U}(\cdot-\alpha)F$
where we recall that $\mathcal{U} $ 
is the fundamental solution of the Stokes system in $\R^3$, see Section \ref{SSE}. We define as well the associated pressure as $p:=\mathcal{P}(\cdot-\alpha)F $.
Thus 
 $(u,p)$ satisfies $\div u= 0 $ and the steady Stokes equation  $-\Delta u+\nabla p =\delta_\alpha F $ on  on $\R^3$. 
 Moreover, 
 \begin{align*}
     u(x) =  \mathcal{U} (x-a)  F 
     = \frac{1}{8\pi} \left(\frac{I}{|x-\alpha|}+\frac{(x-\alpha) \otimes (x-\alpha)}{|x-a|^3} \right)  8\pi|x-\alpha| \left(I-\frac{1}{2}\frac{(x-\alpha) \otimes (x-\alpha)}{|x-\alpha|^2} \right)a = a .
 \end{align*}
In order to regularise the right hand side $\delta_{\alpha}F$ we set $\tilde{u} =\mathcal{U} \star \frac{\chi_{B(\delta,\varepsilon)}}{|B(\delta,\varepsilon)|}F$ with $\varepsilon<r$. One can show that $ |a-\tilde{u}(x)|=|u(x)-\tilde{u}(x)|\leq C_{x,\alpha} \varepsilon$ for $\varepsilon$ small enough and this shows that the set 
$$
\{ u(x),\exists \, f\in L^\infty(\R^3) \text{ compactly supported in } B(\alpha,r) \text{st } -\Delta u+\nabla p =f, \quad \div u= 0  \text{ on } \R^3\}
$$
is dense in $\R^3$, since it is a finite dimensional vector space and hence closed, we get the desired equality. 
\end{proof}

 In the sequel we will look for a solution $ (u,\rho)$ to the transport-Stokes system 
in $\R^3 \setminus \omega$ starting with  $\rho_0$  at time $0$, and exactly reaching the null state   at time $T$, whose main part is given by a rescaling of the auxiliary solution $u_\aux$, while the influence of the initial data will be considered as an error term. 
More precisely, let $\eps>0$, set 
$$u^\eps_\aux(t,x):=\frac{1}{\eps}u_\aux(\frac{t}{\eps},x) \quad \text{ and } \quad 
p^\eps_\aux(t,x):=\frac{1}{\eps} p_\aux(\frac{t}{\eps},x),
$$
which satisfy 
 \begin{equation}
\label{s*}
 -\Delta u^\eps_\aux +\nabla p^\eps_\aux = 0 \quad \text{ in } \quad \R^3 \setminus w, \quad
 \div u^\eps_\aux=0 \quad \text{ in } 
 \R^3 .
\end{equation}
Let us set accordingly
$$X_\eps(t,s,x):=X(t/\eps,s,x),$$
 the associated flow defined on $[0,\eps T]$.

 \subsection{Construction of an error term due to the initial density and reaching equilibrium} 
 This section is devoted to the proof of the following result, where, roughly speaking, it is proved that for $\eps >0$ small enough, one may construct an error term, due to the initial density, which  reaches equilibrium before the imparted time. 
\begin{prop}\label{propRn}
There exists $(u^* , \rho^*) $ in $C([0,+\infty),W^{2,p}(\R^3) \times  L_p(\R^3))$
so that, in  $\R^3$, for $\eps >0$ small enough, 
\begin{equation}
\label{ts*}
\left\{
\begin{array}{rcl}
&\partial_t \rho^* + \div ((u^\eps_\aux+u^*)\rho)=g^* ,\\
&-\Delta u^* + \nabla p = - \rho^* e_3 \quad \text{ and } \div u^*= 0  ,\text{ in } \R^3\\
&\rho^*(0,\cdot)=\rho_0 \quad \text{ and } \quad  \rho^*(\eps T ,\cdot)=0 .
\end{array}
\right.
\end{equation}
with 
 $g^*$ compactly supported in $(0,\eps T)\times \omega$.
\end{prop}
 \begin{proof}[Proof of Proposition \ref{propRn}]
 We introduce the following sequence $(\rho^n,u^n)_{n \geq 0}$ defined on $[0,\eps T]\times \R^3$ such that 
\begin{equation}\label{formule1}
\|\rho^n\|_{L^1\cap L^p} \leq \|\rho_0\|_{L^1\cap L^p} , \|u^n\|_{W^{2,p}} \leq K  ,
\end{equation}
defined as follows.
\begin{itemize}
\item We set $(u^0,\rho^0):=(u_0,\rho_0)$ where $u_0$ the solution to the Stokes equation on $\R^3$ with a right hand side given by $-\rho_0e_3$.
\item Given $(u^n,\rho^n)$ satisfying \eqref{formule1}, we set $\rho^{n+1/2}$ as 
$$
\rho^{n+1/2}:=X_\epsilon^n \# \rho_0
$$
where $X^n_\varepsilon$ the characteristic flow associated to $u^\eps_\aux+u^n $. Then $\rho^{n+1/2}$ satisfies weakly
\begin{eqnarray*}
 \partial_t \rho^{n+1/2}+ \div ((u^\eps_\aux+u^n)\rho^{n+1/2})=0 , {\text{ on } \R^3 }\\
\rho^{n+1/2}(0,\cdot)=\rho_0.
\end{eqnarray*}
\item Using \eqref{forumule_controle_zone} together with \eqref{formule1}, there exists $\epsilon_0=C(T,K,\delta, \Lip(u_\aux))>0$ independent of $n$ such that for all $\varepsilon$ in $(0, \varepsilon_0)$,
\begin{equation}\label{formule2}
 \forall t \in [0,\eps T],\: \forall x\in B(x_i,r_i), \quad  t \in ]t_{i+1/4},t_{i+1/2}[ \Rightarrow
X^n_\eps(t,0,x)\in B(0,3 \delta).
\end{equation}
With the same notations as above, let $\{ \chi_i \}_{1 \leq i \leq L}$ a partition of unity such that $\supp \chi_i \subset B(x_i, r_i)$. 
We set then for $t\in[0,\eps T]$ and $x\in \R^3$, 
$$
\rho^{n+1}(t,x):= \underset{i=1}{\overset{L}{\sum}}\beta_i(t) \rho_0(X^n_\varepsilon(0,t,x)) \chi_i(X^n_\varepsilon(0,t,x))\\
=\underset{i=1}{\overset{L}{\sum}}\beta_i(t) \rho^{n+1/2}(t,x) \chi_i(X^n_\varepsilon(0,t,x)),
$$
with $\beta_i(t) :=\beta\left(8(L+1)(t-t_i) \right) $ with $\beta$ a cutoff function satisfying $\beta_i=1 $ for $t\leq 1/2$ and $\beta_i=0$ for $t\geq 1$.
We observe that 
$\rho^{n+1}$ satisfies weakly
\begin{equation} \label{n-den}
\partial_t \rho^{n+1}+ \div ((u^\eps_\aux+u^n)\rho^{n+1})=g^n \text{ on } \R^3 , \quad 
\rho^{n+1}(0,\cdot)=\rho_0 \quad \text{ and } \quad
\rho^{n+1}(\eps T ,\cdot)=0 ,
\end{equation}
with 
$$
g^n(t,x) := \underset{i=1}{\overset{L}{\sum}}\beta'_i(t) \rho_0(X^n_\varepsilon(0,t,x)) \chi_i(X^n_\varepsilon(0,t,x)) . $$
\item Finally we define $u^{n+1}$ as the unique solution to the Stokes equation:
\begin{equation} \label{n-sto}
-\Delta u^{n+1} + \nabla p^{n+1} =-\rho^{n+1}e_3 \quad \text{ and } \quad \div u^{n+1}= 0 \quad \text{ in } \quad \R^3.
\end{equation}
\end{itemize}

\begin{lemme}\label{lemmeRn}
The vector fields $g^n$ are compactly supported in $(0,\eps T)\times \omega$.
\end{lemme}
\begin{proof}[Proof of Lemma \ref{lemmeRn}]
 If $g^n(t,x)\neq 0$ there exists $1 \leq i \leq L$ such that  $\beta'_i(t) \rho_0(X^n_\varepsilon(0,t,x)) \chi_i(X^n_\varepsilon(0,t,x))\neq 0$ hence $t\in \eps ]t_{i+1/4},t_{i+1/2}[$ with $X_\eps^n(0,t,x)\in \supp \rho_0 \cap B(x_i, r_i)$ which yields $x \in X_\eps^n(t,0,B(x_i,r_i)) \subset \omega $ according to \eqref{formule2}.
\end{proof}


The sequences $\rho^n$ and $g^n$ satisfy the following uniform bounds for all $q\in [0,+\infty]$, 
$$
\|g^n\|_{L^q(0,\eps T;L^1\cap L^p)}+\|\rho^n\|_{L^q(0,\eps T;L^1\cap L^p)}+ \|\partial_t\rho^n\|_{L^q(0,\eps T;L^p+\dot{W}^{-1,p})} \leq M
$$
which yields for $u^n$ using the fact that $-\Delta \partial_t u^n+ \nabla \partial_t p^n= - \partial_t \rho^n e_3$, $\div \partial_t u^n=0$,
\begin{eqnarray*}
\|u^n\|_{L^q(0,\eps T;W^{2,p})}+ \|\partial_t u^n\|_{L^q(0,\eps T;\dot{W}^{1,p})} \leq M, & \text{ if } p>3,\\
\|u^n\|_{L^q(0,\eps T;W_{\loc}^{2,3})}
+ \|\partial_t u^n\|_{L^q(0,\eps T;\dot{W}^{-1,3})} \leq M, & \text{ if } p=3.
\end{eqnarray*}
Hence using the Aubin-Lions Theorem we get the convergence of a subsequence $u^n \to u^* $ in $L^q(0,\eps T;  C^{1,\mu}_{\loc}\cap W^{1,p}_{\loc} (\R^3))$  if $p>3$. This allows us to pass in the limit in the weak formulations of the Stokes equations \eqref{n-sto} and the transport equation \eqref{n-den}, using the weak convergence $\rho^n \rightharpoonup \rho^* $ and $g^n \rightharpoonup  g^* $ in $L^q(0,\eps T;L^p) $. For $p=3$ we get analogously $u^n \to u^* $ in $L^q(0,\eps T,W^{1,3}_{\loc}(\R^3)) $ together with $\rho^n \rightharpoonup  \rho^* $  weakly in $L^q(0,\eps T;L^s) $ for any $1 \leq s \leq 3$ which allows us to pass in the transport equation for $s=3/2$, so that, in  $\R^3$, \eqref{ts*} holds true. 

We emphasize that an Ascoli-type argument similar to the one used in the proof of Theorem 1.1 still holds true for the sequence of flows $X_\epsilon^n $ and we get $X^n_\epsilon(\cdot,s,\cdot) \to X^* (\cdot,s,\cdot) $ in $ C[0,\eps T], C^1_{\loc}) $ (resp. in $ C[0,\eps T], C_{\loc}) $) if $p>3$ (resp. $p=3$). 

This allows us to get a pointwise convergence of $g^n$ to 
$$g^* =\underset{i=1}{\overset{L}{\sum}} \partial_t \beta_i(t) \rho_0(X^*(0,t,x))\chi_i(X^*(0,t,x)),$$
in $(0,\eps T)\times \R^3$ and ensures that $g^*$ is compactly supported in $(0,\eps T)\times \omega$. 
Moreover, passing in the limit, pointwise, in the identity:
$$
\rho^{n+1}(t,x):= \underset{i=1}{\overset{L}{\sum}}\beta_i(t) \rho_0(X^n_\varepsilon(0,t,x)) \chi_i(X^n_\varepsilon(0,t,x)),
$$
we obtain that $\rho^* $ is in $  C(0,T; L^p) $ and the continuity for the velocity $u^*$ as well.
This concludes the proof of Proposition \ref{propRn}. 
\end{proof}


 \subsection{Conclusion of the proof of Theorem \ref{TS-control}} 

Combining  Proposition \ref{lemme1}
and 
Proposition \ref{lemmeRn}, in particular  by \eqref{s*} and \eqref{ts*}, and since $g^*$ is compactly supported in $(0,\eps T)\times \omega$, the vector field
 $(u:= u^\eps_\aux+u^*,\rho :=\rho^*)$ satisfies
\begin{gather*}
\partial_t \rho+ \div (\rho u)= 0  \quad \text{ in } \quad \R^3 \setminus \omega ,
\\  -\Delta u + \nabla  p = - \rho e_3  \quad \text{ in } \R^3 \setminus \omega ,\\
\div u = 0 
 \quad \text{ in } \quad \R^3 ,\\
\rho(0,\cdot)=\rho_0 \quad \text{ and } \quad 
\rho(T\eps ,\cdot)=0 .
\end{gather*}
Recalling the preliminary remark at the beginning of this section, this concludes the proof of Theorem \ref{TS-control}.

\smallskip 
\smallskip 
\smallskip 
\noindent

\subsection*{Acknowledgment}
A.M. and F.S. are supported by the SingFlows project, grant ANR-18-CE40-0027 of the French National Research Agency (ANR).
 This work was partly accomplished while F.S. was participating in a program hosted by the Mathematical Sciences Research Institute in Berkeley, California, during the Spring 2021 semester, and supported by the National Science Foundation under Grant No. DMS-1928930.



 \Addresses
 
\end{document}